\documentclass[10pt,reqno]{amsart}
\pdfoutput=1
\usepackage{adjustbox,array,bigints,cancel,color,comment,extpfeil,mathdots,mathrsfs,mathtools,MnSymbol,multicol,scalerel,setspace,tcolorbox,tikz,tikz-cd,upgreek}
\usepackage[fontsize = 10pt]{fontsize}
\usepackage[left = 2.5cm, right = 2.5cm, top = 2.5cm, bottom = 2.5cm]{geometry}
\usepackage{hyperref}
\usetikzlibrary{patterns}
\numberwithin{equation}{section}

\theoremstyle{plain}
\newtheorem{Cl}[equation]{Claim}

\newtheorem{Lem}[equation]{Lemma}

\newtheorem{Prop}[equation]{Proposition}

\newtheorem{Thm}[equation]{Theorem}

\theoremstyle{definition}

\newtheorem{Defn}[equation]{Definition}

\theoremstyle{remark}
\newtheorem{Rk}[equation]{Remark}

\theoremstyle{plain}
\newtheorem*{Cl*}{Claim}
\newtheorem*{Conj*}{Conjecture}
\newtheorem*{Lem*}{Lemma}
\newtheorem*{Prop*}{Proposition}
\newtheorem*{Q*}{Question}
\newtheorem*{Schol*}{Scholium}
\newtheorem*{SubCl*}{Subclaim}
\newtheorem*{Thm*}{Theorem}

\theoremstyle{definition}
\newtheorem*{Cond*}{Condition}
\newtheorem*{Cstr*}{Construction}
\newtheorem*{Defn*}{Definition}
\newtheorem*{Ex*}{Example}
\newtheorem*{Exs*}{Examples}
\newtheorem*{Md*}{Method}
\newtheorem*{Nt*}{Notation}
\newtheorem*{Pty*}{Property}

\theoremstyle{remark}

\newtheorem*{Rk*}{Remark}
\newtheorem*{Rks*}{Remarks}
\newtheorem*{A-d}{Aside}

\newcommand{\cag}{\begin{equation}\begin{gathered}}
\newcommand{\caag}{\end{gathered}\end{equation}}
\newcommand{\caw}{\begin{equation*}\begin{gathered}}
\newcommand{\caaw}{\end{gathered}\end{equation*}}
\newcommand{\e}{\begin{equation}\begin{aligned}}
\newcommand{\ee}{\end{aligned}\end{equation}}
\newcommand{\ew}{\begin{equation*}\begin{aligned}}
\newcommand{\eew}{\end{aligned}\end{equation*}}

\newcommand{\bcd}{\begin{tikzcd}}
\newcommand{\ecd}{\end{tikzcd}}
\newcommand{\bma}{\begin{matrix}}
\newcommand{\ema}{\end{matrix}}
\newcommand{\bpm}{\begin{pmatrix}}
\newcommand{\epm}{\end{pmatrix}}
\newcommand{\bvm}{\begin{vmatrix}}
\newcommand{\evm}{\end{vmatrix}}

\newcommand{\nts}{\begin{tcolorbox}}
\newcommand{\ntss}{\end{tcolorbox}}

\newcommand{\cref}[1]{Corollary \ref{#1}}

\newcommand{\dref}[1]{Definition \ref{#1}}

\newcommand{\eref}[1]{eqn.\hspace{0.6mm}(\ref{#1})}

\newcommand{\lref}[1]{Lemma \ref{#1}}

\newcommand{\pref}[1]{Proposition \ref{#1}}

\newcommand{\rref}[1]{Remark \ref{#1}}

\newcommand{\sref}[1]{\S\ref{#1}}
\newcommand{\srefs}[1]{\S\S\ref{#1}}
\newcommand{\tref}[1]{Theorem \ref{#1}}

\DeclareMathSizes{10}{10}{8}{7}

\newcommand{\msm}[1]{\mbox{\small \(#1\)}}
\newcommand{\mns}[1]{\mbox{\normalsize \(#1\)}}

\newcommand{\mLa}[1]{\mbox{\Large \(#1\)}}

\newcommand{\bb}[1]{\mathbb{#1}}
\newcommand{\cal}[1]{\mathscr{#1}}
\newcommand{\fr}[1]{\mathfrak{#1}}
\newcommand{\mb}[1]{\mbox{\boldmath \(#1\)}}
\newcommand{\mc}[1]{\mathcal{#1}}

\newcommand{\abo}{\hspace{3mm}\text{or}\hspace{3mm}}

\newcommand{\et}{\hspace{3mm}\text{and}\hspace{3mm}}
\newcommand{\hs}[1]{\hspace{#1}}

\newcommand{\vs}[1]{\vspace{#1}}

\DeclareMathSymbol{\Alpha}{\mathalpha}{operators}{"41}
\DeclareMathSymbol{\Beta}{\mathalpha}{operators}{"42}
\DeclareMathSymbol{\Epsilon}{\mathalpha}{operators}{"45}
\DeclareMathSymbol{\Zeta}{\mathalpha}{operators}{"5A}
\DeclareMathSymbol{\Eta}{\mathalpha}{operators}{"48}
\DeclareMathSymbol{\Iota}{\mathalpha}{operators}{"49}
\DeclareMathSymbol{\Kappa}{\mathalpha}{operators}{"4B}
\DeclareMathSymbol{\Mu}{\mathalpha}{operators}{"4D}
\DeclareMathSymbol{\Nu}{\mathalpha}{operators}{"4E}
\DeclareMathSymbol{\Omicron}{\mathalpha}{operators}{"4F}
\DeclareMathSymbol{\Rho}{\mathalpha}{operators}{"50}
\DeclareMathSymbol{\Tau}{\mathalpha}{operators}{"54}
\DeclareMathSymbol{\Chi}{\mathalpha}{operators}{"58}
\DeclareMathSymbol{\omicron}{\mathord}{letters}{"6F}

\newcommand{\al}{\alpha}

\renewcommand{\th}{\theta}

\newcommand{\io}{\iota}
\newcommand{\ka}{\kappa}
\newcommand{\la}{\lambda}
\newcommand{\rh}{\rho}

\newcommand{\si}{\sigma}

\newcommand{\ta}{\tau}

\newcommand{\om}{\omega}

\newcommand{\Ga}{\Gamma}

\newcommand{\La}{\Lambda}
\newcommand{\Si}{\Sigma}

\newcommand{\Om}{\Omega}

\newcommand{\<}{\langle}
\newcommand{\?}{\rangle}

\newcommand{\Ann}{\operatorname{Ann}}

\newcommand{\Ds}{\bigoplus}
\newcommand{\ds}{\oplus}

\newcommand{\Hom}{\operatorname{Hom}}
\newcommand{\Id}{\operatorname{Id}}
\newcommand{\Ker}{\operatorname{Ker}}

\newcommand{\lqt}[2]{\left.\raisebox{-1mm}{\(#2\)}\middle\backslash\raisebox{1mm}{\(#1\)}\right.}

\newcommand{\rqt}[2]{\left.\raisebox{1mm}{\(#1\)}\middle/\raisebox{-1mm}{\(#2\)}\right.}
\newcommand{\ts}{\otimes}

\newcommand{\Tr}{\operatorname{Tr}}

\newcommand{\x}{\times}

\newcommand{\del}{\partial}

\newcommand{\Hocl}[1]{\overset{\circ}{H^{#1}}_{\kern-1.9mm\cl}}

\newcommand{\lop}{\left\|\kern-1.30mm\left\|}
\newcommand{\op}{\|\kern-1.30mm\|}
\newcommand{\rop}{\right\|\kern-1.30mm\right\|}
\newcommand{\SI}{\operatorname{\cal{I}\kern-1.5pt nd}}

\newcommand{\cc}{\subseteq}

\newcommand{\es}{\emptyset}

\newcommand{\mt}{\mapsto}

\newcommand{\osr}{\backslash}
\newcommand{\oto}[1]{\xrightarrow{#1}}
\newcommand{\pc}{\subset}

\newcommand{\yy}{\supseteq}

\newcommand{\CL}{\mathcal{C}l}
\newcommand{\cl}{\mathrm{closed}}

\newcommand{\Conv}{\mathrm{Conv}}

\newcommand{\dd}{\mathrm{d}}

\newcommand{\dR}[1]{H^{#1}_{\operatorname{dR}}}

\newcommand{\emb}{\hookrightarrow}

\newcommand{\Gr}{\mathrm{Gr}}

\newcommand{\hk}{\righthalfcup}

\newcommand{\M}{\mathrm{M}}

\newcommand{\oGr}{\widetilde{\mathrm{\Gr}}}

\newcommand{\Op}{\mathcal{O}p}

\newcommand{\T}{\mathrm{T}}

\newcommand{\w}{\wedge}
\newcommand{\ww}[2][{}]{\bigwedge{\hspace{-1mm}}^{#2}_{\hspace{1mm}#1}\hspace{0.1mm}}

\newcommand{\0}{\infty}
\newcommand{\1}{\cdot}
\newcommand{\bin}{\binom}

\renewcommand{\ge}{\geqslant}
\newcommand{\gl}{\hspace{0.4mm}\raisebox{0.8mm}{\(>\)}\kern-1.8mm\raisebox{-0.8mm}{\(<\)}\hspace{0.4mm}}
\newcommand{\gle}{\hspace{0.4mm}\raisebox{1.2mm}{\(\ge\)}\kern-1.8mm\raisebox{-1.2mm}{\(\le\)}\hspace{0.4mm}}
\renewcommand{\le}{\leqslant}
\newcommand{\pt}{\bullet}

\newcommand{\GL}{\operatorname{GL}}

\newcommand{\sg}{\(\widetilde{\mathrm{G}}_2\)}
\newcommand{\SL}{\operatorname{SL}}

\newcommand{\slr}{\(\operatorname{SL}(3;\mathbb{R})^2\)}

\newcommand{\Stab}{\operatorname{Stab}}

\renewcommand{\iff}{if and only if}

\newcommand{\Wlg}{Without loss of generality}
\newcommand{\wlg}{without loss of generality}

\newcommand{\wrt}{with respect to}

\newcommand{\lt}{\left}
\newcommand{\m}{\middle}

\newcommand{\rt}{\right}

\newcommand{\tld}{\widetilde}

\title[The relative \(\lowercase{h}\)-principle for closed \(\SL(3;\bb{R})^2\) 3-forms]{The relative \(\mb{\lowercase{h}}\)-principle for closed \(\mb{\SL(3;\bb{R})^2}\) 3-forms}
\author{Laurence H. Mayther}

\begin{document}\fontsize{10pt}{12pt}\selectfont
\begin{abstract}
\footnotesize{This paper uses convex integration with avoidance and transversality arguments to prove the relative \(h\)-principle for closed \slr\ 3-forms on oriented 6-manifolds.  As corollaries, it is proven that if an oriented 6-manifold \(\M\) admits any \slr\ 3-form, then every degree 3 cohomology class on \(\M\) can be represented by an \slr\ 3-form and, moreover, that the corresponding Hitchin functional on \slr\ 3-forms representing this class is necessarily unbounded above.  Essential to the proof of the \(h\)-principle is a careful analysis of the rank 3 distributions induced by an \slr\ 3-form and their interaction with generic pairs of hyperplanes.  The proof also introduces a new property of sets in affine space, termed macilence, as a method of verifying ampleness.}
\end{abstract}
\maketitle

\section{Introduction}

This is the second of two papers by the author which seek to investigate which classes of closed stable forms satisfy the relative \(h\)-principle.  In \cite{RhPfCSF}, the author used classical convex integration to prove the relative \(h\)-principle for stable \((2k-2)\)-forms in \(2k\) dimensions, \((2k-1)\)-forms in \(2k+1\) dimensions, \sg\ 3-forms and \sg\ 4-forms, each of which had previously not been known to satisfy the relative \(h\)-principle.  The purpose of the current paper is to examine a further class of stable forms where the relative \(h\)-principle had previously not been known to hold, {\it viz.}\ \slr\ 3-forms, for which different methods are required.  By applying a special case of Gromov's general theory of convex integration via convex hull extensions, known as convex integration with avoidance (recently introduced in \cite{CIwA&H46D}), I prove that the relative \(h\)-principle holds in the \slr\ case.  I begin by recounting some notation.

Let \(\lt(\th^1,...,\th^6\rt)\) denote the standard basis of \(\lt(\bb{R}^6\rt)^*\) and define:
\ew
\rh_+ = \th^{123} + \th^{456},
\eew
where multi-index notation \(\th^{ij...k} = \th^i \w \th^j \w ... \w \th^k\) is used throughout this paper.  Given an oriented 6-manifold \(\M\), a 3-form \(\rh\) on \(\M\) is termed an \slr\ 3-form if for all \(x \in \M\), there exists an orientation-preserving isomorphism \(\al: \T_x\M \to \bb{R}^6\) such that \(\rh|_x = \al^*\rh_+\).  The name is motivated by the observation that the stabiliser of \(\rh_+\) in \(\GL_+(6;\bb{R})\) is isomorphic to \slr\ acting diagonally; thus, \slr\ 3-forms on \(\M\) are in bijective correspondence with \slr-structures, i.e.\ principal \slr-subbundles of the oriented frame bundle of \(\M\).  Since the \(\GL_+(6;\bb{R})\)-orbit of \(\rh_+\) in \(\ww{3}\lt(\bb{R}^6\rt)^*\) is open, \slr\ 3-forms are stable (as defined in \cite{SF&SM}) and thus all sufficiently small perturbations of an \slr\ 3-form are also of \slr-type.  Write \(\ww[+]{3}\T^*\M\) for the bundle of \slr\ 3-forms over \(\M\) and \(\Om^3_+\) for the corresponding sheaf of sections.

Write \(\CL^3_+(\M)\) for the set of closed \slr\ 3-forms on \(\M\) and, given a fixed cohomology class \(\al \in \dR{3}(\M)\), write \(\CL^3_+(\al)\) for the set of closed \slr\ 3-forms representing the class \(\al\).  More generally, given a submanifold \(A \pc \M\) (or polyhedron; see \sref{pre:slr}), let \(\rh_r\) be a closed \slr\ 3-form on \(\Op(A)\) such that \([\rh_r] = \al|_{\Op(A)} \in \dR{3}(\Op(A))\) and write:
\caw
\Om^3_+(\M;\rh_r) = \lt\{ \rh \in \Om^3_+(\M) ~\m|~ \rh|_{\Op(A)} = \rh_r\rt\};\\
\CL^3_+(\M;\rh_r) = \lt\{ \rh \in \Om^3_+(\M;\rh_r) ~\m|~ \dd\rh = 0 \rt\};\\
\CL^3_+(\al;\rh_r) = \lt\{ \rh \in \CL^3_+(\M;\rh_r) ~\m|~ [\rh] = \al \in \dR{3}(\M)\rt\}.
\caaw
For the purposes of simplicity, say that \slr\ 3-forms satisfy the relative \(h\)-principle if for every \(\M\), \(A\), \(\al\) and \(\rh_r\), the inclusions:
\ew
\CL^3_+(\al;\rh_r) \emb \CL^3_+(\M;\rh_r) \emb \Om^3_+(\M; \rh_r)
\eew
are homotopy equivalences -- although the reader should note that a slightly stronger definition of \(h\)-principle is used in the main body of this paper; see \sref{pre:slr} for details.  The main theorem of this paper is the following.

\begin{Thm}\label{slr-hP-thm}
\slr\ 3-forms satisfy the relative \(h\)-principle.  In particular, taking \(A = \es\) in the definition of the relative \(h\)-principle, the inclusions:
\ew
\CL^3_+(\al) \emb \CL^3_+(\M) \emb \Om^3_+(\M)
\eew
are homotopy equivalences and thus if \(\M\) admits any \slr\ 3-form, then every degree 3 cohomology class on \(\M\) can be represented by an \slr\ 3-form.
\end{Thm}

As an application of \tref{slr-hP-thm}, recall that, since \(\SL(3;\bb{R})^2 \pc \SL(6;\bb{R})\), there is a natural Hitchin functional \(\mc{H}: \CL^3_+(\al) \to (0,\0)\) defined whenever \(\CL^3_+(\al) \ne \es\) (see \sref{pre:slr} for details).  By combining \tref{slr-hP-thm} with \cite[Thm.\ 4.1]{RhPfCSF}, one obtains:

\begin{Thm}
Let \(\M\) be any closed, oriented 6-manifold admitting \slr\ 3-forms.  Then, for each \(\al \in \dR{3}(\M)\), \(\CL^3_+(\al) \ne \es\) and the functional:
\ew
\mc{H}: \CL^3_+(\al) \to (0,\infty)
\eew
is unbounded above.  More generally, if \(\M\) is a closed, oriented \(6\)-orbifold and \(\CL^3_+(\al) \ne \es\), then the same conclusion applies.
\end{Thm}

The proof of \tref{slr-hP-thm} builds on the observation, taken from \cite[Lem.\ 5.2]{RhPfCSF}, that in order to prove the relative \(h\)-principle for \slr\ 3-forms, it suffices to prove the classical relative \(h\)-principle, as described in \cite[\S6.2]{ItthP}, for a family of fibred differential relations \(\cal{R}_+(a)\) defined explicitly in \sref{DR-form} (where \(a\) ranges over all possible continuous maps \(a:D^q \to \Om^3(\M)\) for all possible values of \(q \ge 0\)).  Crucially, however, unlike the relations considered in \cite{RhPfCSF}, the relation \(\cal{R}_+(a)\) is not ample and thus the \(h\)-principle for \(\cal{R}_+(a)\) cannot be proven using convex integration.  Instead, recall that a subset \(A\) of an affine space \(\bb{A}\) is termed ample if the convex hull of each path component of \(A\) is equal to \(\bb{A}\).  Given a point \(x \in \M\), a hyperplane \(\bb{B} \pc \T_x\M\) and an \slr\ 3-form \(\rh \in \ww[+]{3}\T^*_x\M\), \(\cal{R}_+(a)\) defines a subspace \(\mc{N}(\rh;\bb{B})_0 \pc \ww{2}\bb{B}^*\) (see \sref{Defn-AT}).  Whilst \(\mc{N}(\rh;\bb{B})_0 \pc \ww{2}\bb{B}^*\) is not ample for all \(\rh\) and \(\bb{B}\), for each fixed \(\rh\) the set \(\mc{N}(\rh;\bb{B})_0\) is ample for generic choices of \(\bb{B}\).  Thus, informally, the relations \(\cal{R}_+(a)\) are `close' to being ample, and hence the \(h\)-principle for the relations \(\cal{R}_+(a)\) can be proven using convex integration with avoidance.  The main task in this paper, therefore, lies in defining a suitable notion of when a hyperplane \(\bb{B}\) (and, more generally, when a finite set of distinct hyperplanes \(\Xi\)) is generic \wrt\ a given \slr\ 3-form \(\rh\), and verifying that generic hyperplanes have the necessary properties to enable convex integration with avoidance to be applied.  Specifically, it must be proven that given an \slr\ 3-form \(\rh \in \ww[+]{3}\T^*_x\M\) and a generic set \(\Xi\) of hyperplanes, \(\Xi\) is generic for `almost all' \slr\ 3-forms \(\rh'\) which have the same tangential component along \(\bb{B}\) as \(\rh\) (\lref{mac-lem}).  Establishing this fact forms the technical heart of this paper and relies on a careful analysis of the rank 3 distributions induced by an \slr\ 3-form and their interaction with generic pairs of hyperplanes (see \srefs{Epm-Deriv-Sec}--\ref{3rd-Mac-Sec}).

The results of this paper were obtained during the author's doctoral studies, which where supported by EPSRC Studentship 2261110.\\

\section{Preliminaries}

\subsection{\(\mb{\SL(3;\bb{R})^2}\) 3-forms}\label{pre:slr}

Let \(\M\) be an oriented 6-manifold and let \(\rh \in \Om^3(\M)\).  Define a homomorphism \(K_\rh: \T\M \to \T\M \ts \ww{6}\T^*\M\) by composing the map:
\ew
\bcd[row sep = 0pt]
\T\M \ar[r]& \ww{5}\T^*\M\\
v \in \T_x\M \ar[r, maps to]& (v\hk\rh|_x)\w\rh|_x
\ecd
\eew
with the canonical isomorphism \(\ww{5}\T^*\M \cong \T\M \ts \ww{6}\T^*\M\).  Define a section \(\La(\rh)\) of \(\lt(\ww{6}\T^*\M\rt)^2\) by:
\ew
\La(\rh) = \frac{1}{6}\Tr\lt(K_\rh^2\rt),
\eew
where \(\Tr\) denotes fibrewise trace.  It can be shown \cite{TGo3Fi6&7D} that \(\rh\) is an \slr\ 3-form \iff\ \(\La(\rh) > 0\) (recall that \(\lt(\ww{6}\T^*\M\rt)^2\) is naturally oriented by declaring \(s \ts s > 0\) for any non-zero \(s \in \ww{6}\T^*\M\)).  In particular, \(\rh\) induces a volume form \(vol_\rh\) on \(\M\) via the formula:
\ew
vol_\rh = \lt(\La(\rh)\rt)^\frac{1}{2}.
\eew
In the specific case where \(\M\) is closed, for each cohomology class \(\al \in \dR{3}(\M)\) one may consider the Hitchin functional:
\ew
\bcd[row sep = 0pt]
\mc{H}: \CL^3_+(\al)  \ar[r] & (0,\infty)\\
\rh \ar[r, maps to] & \bigint_\M vol_\rh
\ecd
\eew
whenever \(\CL^3_+(\al) \ne \es\), as defined in \cite{TGo3Fi6&7D}.

For an arbitrary manifold \(\M\), \(\rh\) also induces a para-complex structure \(I_\rh = vol_\rh^{-1}K_\rh\) on \(\M\), i.e.\ \(I_\rh\) is an endomorphism of \(\T\M\) satisfying \(I_\rh^2 = \Id\), such that the \(\pm1\)-eigenbundles of \(I_\rh\), denoted \(E_{\pm,\rh}\), are each rank 3.  For later calculations in this paper it is useful to note that, for the `standard' \slr\ 3-form \(\rh_+\) on \(\bb{R}^6\), the above constructions yield:
\caw
vol_{\rh_+} = \th^{123456}, \hs{5mm}  I_{\rh_+} = (e_1,e_2,e_3,e_4,e_5,e_6) \mt (e_1,e_2,e_3,-e_4,-e_5,-e_6),\\
E_+ = \<e_1,e_2,e_3\? \et E_- = \<e_4,e_5,e_6\?,
\caaw
where \((e_i)_i\) denotes the canonical basis of \(\bb{R}^6\).

Next, recall that a (possibly disconnected) subset \(A\cc\M\) is termed a polyhedron if there exists a smooth triangulation \(\cal{K}\) of \(\M\) identifying \(A\) with a subcomplex of \(\cal{K}\) (in particular, \(A\) is a closed subset of \(\M\)); examples of polyhedra include disjoint unions of submanifolds of \(\M\).  Following \cite{PDR}, write \(\Op(A)\) for an arbitrarily small but unspecified open neighbourhood of \(A\) in \(\M\), which may be shrunk whenever necessary.  Let \(D^q\) denote the \(q\)-dimensional disc (\(q \ge 0\)), let \(\al: D^q \to \dR{3}(\M)\) be a continuous map and let \(\fr{F}_0: D^q \to \Om^3_{+}(\M)\) be a continuous map such that:
\begin{enumerate}
\item For all \(s \in \del D^q\): \(\dd\fr{F}_0(s) = 0\) and \([\fr{F}_0(s)] = \al(s) \in \dR{3}(\M)\);
\item For all \(s \in D^q\): \(\dd\lt(\fr{F}_0(s)|_{\Op(A)}\rt) = 0\) and \(\lt[\fr{F}_0(s)|_{\Op(A)}\rt] = \al(s)|_{\Op(A)} \in \dR{3}(\Op(A))\).
\end{enumerate}
(Note that, since all sufficiently small open neighbourhoods of \(A\) in \(\M\) deformation retract onto \(A\), (2) is independent of the choice of \(\Op(A)\).)  As in the author's recent paper \cite{RhPfCSF}, say that \slr\ 3-forms satisfy the relative \(h\)-principle if for every \(\M\), \(A\), \(q\), \(\al\) and \(\fr{F}_0\) as above, there exists a homotopy \(\fr{F}_\pt: [0,1] \x D^q \to \Om^3_+(\M)\), constant over \(\del D^q\), satisfying:
\begin{enumerate}
\setcounter{enumi}{2}
\item For all \(s \in D^q\) and \(t \in [0,1]\): \(\fr{F}_t(s)|_{\Op(A)} = \fr{F}_0(s)|_{\Op(A)}\);
\item For all \(s \in D^q\): \(\dd\fr{F}_1(s) = 0\) and \([\fr{F}_1(s)] = \al(s) \in \dR{3}(\M)\).
\end{enumerate}
Given that \slr\ 3-forms satisfy the relative \(h\)-principle, standard homotopy-theoretic arguments (as in \cite[\S6.2.A]{ItthP}) show that the inclusions:
\ew
\CL^3_+(\al;\rh_r) \emb \CL^3_+(\M;\rh_r) \emb \Om^3_+(\M; \rh_r)
\eew
are homotopy equivalences, for any choice of \(\M\), \(A\), \(\al\) and \(\rh_r\).  Thus, the above definition is consistent with (and indeed stronger than) the notion of relative \(h\)-principle described in the introduction.\\

\subsection{Some generalities on stable forms}

For the purposes of this subsection, let \(1 \le p \le n\) and let \(\si_0\) be any stable \(p\)-form on \(\bb{R}^n\), i.e.\ any \(p\)-form such that \(\GL_+(n;\bb{R}) \1 \si_0 \pc \ww{p}\lt(\bb{R}^n\rt)^*\) is open.  Given an oriented \(n\)-dimensional real vector space \(\bb{A}\), write \(\ww[\si_0]{p}\bb{A}^*\) for the set of \(\si_0\)-forms on \(\bb{A}\) where, by analogy with the definition of \slr\ 3-forms, \(\si \in \ww{p}\bb{A}^*\) is called a \(\si_0\)-form if there exists an orientation-preserving isomorphism \(\al: \bb{A} \to \bb{R}^n\) such that \(\al^*\si_0 = \si\).  As in \cite{RhPfCSF}, given \(\ta \in \ww{p}\lt(\bb{R}^{n-1}\rt)^*\) define:
\ew
\mc{N}_{\si_0}(\ta) = \lt\{\nu \in \ww{p-1}\lt(\bb{R}^{n-1}\rt)^* ~\m|~ \th \w \nu + \ta \in \ww[\si_0]{p}\lt(\bb{R} \ds \bb{R}^{n-1}\rt)^* \rt\} \pc \ww{p-1}\lt(\bb{R}^{n-1}\rt)^*,
\eew
where \(\th\) is the standard annihilator of \(\bb{R}^{n-1} \pc \bb{R} \ds \bb{R}^{n-1}\).  The aim of this subsection is to briefly recall some key properties of the set \(\mc{N}_{\si_0}(\ta)\).

Let \(Emb\lt(\bb{R}^{n-1},\bb{R}^n\rt)\) denote the space of linear embeddings \(\io:\bb{R}^{n-1} \to \bb{R}^n\) and consider the map:
\ew
\bcd[row sep = 0pt]
\cal{T}_{\si_0}: Emb\lt(\bb{R}^{n-1},\bb{R}^n\rt) \ar[r] & \ww{p}\lt(\bb{R}^{n-1}\rt)^*\\
\io \ar[r, maps to] & \io^*\si_0\hs{1.5pt}.
\ecd
\eew
\(\GL_+(n-1;\bb{R})\) acts on \(Emb\lt(\bb{R}^{n-1},\bb{R}^n\rt)\) via pre-composition, and the quotient \(\rqt{Emb\lt(\bb{R}^{n-1},\bb{R}^n\rt)}{\GL_+(n-1;\bb{R})}\) may naturally be identified with the oriented Grassmannian \(\oGr_{n-1}\lt(\bb{R}^n\rt)\).  Given \(f\in\GL_+(n-1;\bb{R})\), a direct computation shows:
\ew
\cal{T}_{\si_0}(\io \circ f) = f^*\io^*\lt(\si_0\rt) = f^*\cal{T}_{\si_0}(\io).
\eew
Thus, \(\cal{T}_{\si_0}\) descends to a map \(\oGr_{n-1}\lt(\bb{R}^n\rt) \to \rqt{\ww{p}\lt(\bb{R}^{n-1}\rt)^*}{\GL_+(n-1;\bb{R})}\).  Write \(\mc{S}(\si_0)\) for the stabiliser of \(\si_0\) in \(\GL_+(n;\bb{R})\) and note that \(\mc{S}(\si_0)\) acts on \(Emb\lt(\bb{R}^{n-1},\bb{R}^n\rt)\) (and hence on \(\oGr_{n-1}\lt(\bb{R}^n\rt)\)) on the left via post-composition.  Clearly, \(\cal{T}_{\si_0}\) is invariant under this action and thus \(\cal{T}_{\si_0}\) descends further to a map:
\ew
\bcd
\mc{T}_{\si_0}:\lqt{\oGr_{n-1}\lt(\bb{R}^n\rt)}{\mc{S}(\si_0)} \ar[r] & \rqt{\ww{p}\lt(\bb{R}^{n-1}\rt)^*}{\GL_+(n-1,\bb{R})}.
\ecd
\eew
The following two results will be utilised in the proof of \tref{slr-hP-thm}.

\begin{Prop}[{\cite[Prop.\ 6.2]{RhPfCSF}}]\label{open-to-open}
Let \(\si_0 \in \ww{p}\lt(\bb{R}^n\rt)^*\) be stable and equip the spaces \(\lqt{\oGr_{n-1}\lt(\bb{R}^n\rt)}{\mc{S}(\si_0)}\) and \(\rqt{\ww{p}\lt(\bb{R}^{n-1}\rt)^*}{\GL_+(n-1,\bb{R})}\) with their natural quotient topologies.  Then, \(\mc{T}_{\si_0}\) is an open map.  In particular, if \(\mc{O} \in \lqt{\oGr_{n-1}\lt(\bb{R}^n\rt)}{\mc{S}(\si_0)}\) is an open orbit, then \(\mc{T}_{\si_0}(\mc{O})\) is also an open orbit, i.e.\ the orbit of a stable p-form on \(\bb{R}^{n-1}\).
\end{Prop}

\begin{Lem}[{\cite[Prop.\ 6.4 and Lems.\ 6.7, 6.8 \& 6.9]{RhPfCSF}}]\label{ab-lem}
Suppose there exists an orientation-reversing automorphism \(F \in \GL(n;\bb{R})\) such that \(F^*\si_0 = \si_0\).  If \(\mc{O} \in \lqt{\oGr_{n-1}(\bb{R}^n)}{\mc{S}(\si_0)}\) satisfies \(\mc{T}_{\si_0}^{-1}(\{\mc{T}_{\si_0}(\mc{O})\}) = \{\mc{O}\}\) and moreover if the stabiliser in \(\GL_+(n-1;\bb{R})\) of some (equivalently every) \(\ta \in \mc{T}_{\si_0}(\mc{O})\) is connected, then for all \(\ta \in \mc{T}_{\si_0}(\mc{O})\), the space \(\mc{N}_{\si_0}(\ta) \pc \ww{p-1}\lt(\bb{R}^{n-1}\rt)^*\) is path-connected and:
\ew
\Conv(\mc{N}_{\si_0}(\ta)) = \ww{p-1}\lt(\bb{R}^{n-1}\rt)^*,
\eew
where \(\Conv\) denotes the convex hull.\\
\end{Lem}

\subsection{Configuration spaces for hyperplanes}\label{conf-intro}

This is the first of two subsections which recount convex integration with avoidance, introduced in \cite{CIwA&H46D} (although note that the presentation and notation used below differs from that in \cite{CIwA&H46D}).  Let \(\bb{A}\) be an \(n\)-dimensional vector space and write \(\Gr_{n-1}^{(\infty)}(\bb{A})\) for the collection of all finite subsets of \(\Gr_{n-1}(\bb{A})\).  \(\Gr_{n-1}^{(\infty)}(\bb{A})\) is termed the configuration space for hyperplanes in \(\bb{A}\) and can be given a natural `smooth structure' as follows.  For any \(k\ge 1\), consider the manifold \(\prod_1^k \Gr_{n-1}(\bb{A})\) parameterising ordered \(k\)-tuples of hyperplanes in \(\bb{A}\). The symmetric group \(Sym_k\) acts on \(\prod_1^k \Gr_{n-1}(\bb{A})\) by permuting the factors, however this action is not free and thus the resulting quotient is not a smooth manifold, but rather an orbifold. Now define the subset:
\ew
\lt(\prod_1^k \Gr_{n-1}(\bb{A})\rt)_{sing} = \lt\{(\bb{B}_1,...,\bb{B}_k) \in \prod_1^k \Gr_{n-1}(\bb{A}) ~\m|~ \bb{B}_i = \bb{B}_j \text{ for some \(i\ne j\)} \rt\}
\eew
of tuples whose elements are not distinct. This set consists precisely of those elements of \(\prod_1^k \Gr_{n-1}(\bb{A})\) with a non-trivial stabiliser in \(Sym_k\) and may naturally be regarded as a stratified submanifold of \(\prod_1^k \Gr_{n-1}(\bb{A})\) of codimension \(n-1 = \dim\Gr_{n-1}(\bb{A})\). The complement of this set:
\ew
\tld{\mns{\prod_1^k} \Gr_{n-1}(\bb{A})} = \lt.\prod_1^k \Gr_{n-1}(\bb{A})\middle\backslash\lt(\prod_1^k \Gr_{n-1}(\bb{A})\rt)_{sing}\rt.
\eew
is thus an open and dense subset of \(\prod_1^k \Gr_{n-1}(\bb{A})\) on which the group \(Sym_k\) acts freely. In particular, the space \(\rqt{\tld{\prod_1^k \Gr_{n-1}(\bb{A})}}{Sym_k}\) is naturally a smooth manifold. Denote this manifold by \(\Gr_{n-1}^{(k)}(\bb{A})\) and denote the natural quotient map by \(\si:\tld{\prod_1^k \Gr_{n-1}(\bb{A})} \to \Gr_{n-1}^{(k)}(\bb{A})\).  Since \(\Gr_{n-1}^{(\infty)}(\bb{A}) = \coprod_{k=1}^\infty \Gr_{n-1}^{(k)}(\bb{A})\) as sets, \(\Gr_{n-1}^{(\infty)}(\bb{A})\) inherits a natural topology such that each connected component is a smooth manifold.\\

\subsection{Convex integration with avoidance}\label{ISoCIwA}

Let \(\pi: E \to \M\) be a vector bundle.  Write \(E^{(1)}\) for the first jet bundle of \(E\); explicitly, given a connection \(\nabla\) on \(E\), by \cite[\S9, Cor.\ to Thm.\ 7]{DOoVB} one can identify \(E^{(1)} \cong E \ds \lt(\T^*\M \ts E\rt)\) such that the following diagram commutes:
\ew
\bcd[row sep = 3mm]
\Ga\lt(\M,E^{(1)}\rt) \ar[rr, "\mLa{\cong}"] & & \Ga\lt(\M,E \ds \lt(\T^*\M \ts E\rt)\rt)\\
& \Ga(\M,E) \ar[ru, "\mLa{s \mt s \ds \nabla s}"'] \ar[lu, "\mLa{j_1}"] &
\ecd
\eew
where \(\Ga(\M,-)\) denotes the space of sections over \(\M\) and \(j_1\) is the map assigning to a section of \(E\) its corresponding \(1\)-jet; write \(p_1: E^{(1)} \to E\) for the natural projection.  In particular, note that \(E^{(1)}\) naturally has the structure of a vector bundle over \(\M\).  More generally, given \(q \ge 0\), write \(E_{D^q}\) for the pullback of the vector bundle \(E\) along the projection \(D^q \x \M \to \M\); explicitly, \(E_{D^q}\) is the vector bundle \(D^q \x E \oto{\Id \x \pi} D^q \x \M\).  In this paper, a section of \(E_{D^q}\) shall refer to a continuous map \(s: D^q \x \M \to D^q \x E\) satisfying \(\pi_{E_{D^q}} \circ s = \Id_{D^q \x \M}\) and depending smoothly on \(x \in \M\); in particular, sections of \(E_{D^q}\) over \(D^q \x \M\) correspond to continuous maps \(D^q \to \Ga(E,\M)\).  Write \(E^{(1)}_{D^q}\) for the vector bundle \(\lt(E^{(1)}\rt)_{D^q}\) and note that \(E^{(1)}_{D^q} \ne \lt(E_{D^q}\rt)^{(1)}\), since only derivatives in the `\(\M\)-direction' are considered in the bundle \(E^{(1)}_{D^q}\).  A section of \(E^{(1)}_{D^q}\) is termed holonomic if it is the \(1\)-jet of a section of \(E_{D^q}\), i.e.\ if it can be written as \(s \ds \nabla s\) for some section \(s\) of \(E_{D^q}\).  Now write \(p_1\) for the projection \(E^{(1)} \cong E \ds \lt(\T^*\M \ts E\rt) \to E\) and fix \(x \in \M\).  For any \(e \in E_x\), the fibre of the map \(p_1:E^{(1)} \to E\) over \(e\) is the space \(p_1^{-1}(e) = \{e\} \x \T^*_x\M \ts E_x \cong \{e\} \x \Hom(\T_x\M,E_x)\).  Each codimension-1 hyperplane \(\bb{B} \pc \T_x\M\) and linear map \(\la:\bb{B}\to E_x\) defines a so-called principal subspace of \(p_1^{-1}(e)\), given by:
\e\label{PiBl-defn}
\Pi_e(\bb{B},\la) &= \{e\} \x \lt\{L\in\Hom(\T_x\M,E_x)~\middle|~L|_\bb{B} = \la\rt\}\\
&= \{e\} \x \Pi(\bb{B},\la).
\ee
\(\Pi_e(\bb{B},\la)\) is an affine subspace of \(p_1^{-1}(e)\) modelled on \(E_x\) (though not, in general, a linear subspace; note also that changing the choice of connection changes the identification \(p_1^{-1}(e) = \{e\} \x \T^*_p\M \ts E_p\) by an affine linear map and so the collection of principal subspaces of \(p_1^{-1}(e)\) is independent of the choice of connection).

A fibred differential relation (of order 1) on \(D^q\)-indexed families of sections of \(E\) is simply a subset \(\cal{R} \cc E^{(1)}_{D^q}\).  \(\cal{R}\) is termed an open relation if it is open as a subset of \(E^{(1)}_{D^q}\).  Say that a fibred relation \(\cal{R}\) satisfies the relative \(h\)-principle if for every polyhedron \(A\) and every section \(F_0\) of \(\cal{R}\) over \(D^q \x \M\) which is holonomic over \(\lt(\del D^q \x \M\rt) \cup \lt(D^q \x \Op(A)\rt)\), there exists a homotopy \((F_t)_{t \in [0,1]}\) of sections of \(\cal{R}\), constant over \(\lt(\del D^q \x \M\rt) \cup \lt(D^q \x \Op(A)\rt)\), such that \(F_1\) is a holonomic section of \(\cal{R}\).  (The reader will note the similarity between this definition and the notion of the relative \(h\)-principle for \slr\ 3-forms stated in \sref{pre:slr}.)

Now, consider the vector bundles \(\T\M\) over \(\M\) and \(\T\M_{D^q}\) over \(D^q \x \M\).  Applying the construction of \sref{conf-intro} to each fibre of these vector bundles yields bundles \(\Gr_{n-1}^{(\0)}(\T\M)\) and \(\Gr_{n-1}^{(\0)}(\T\M_{D^q})\) over \(\M\) and \(D^q \x \M\) respectively (note that \(\Gr_{n-1}^{(\0)}(\T\M_{D^q})\) is simply the bundle \(D^q \x \Gr_{n-1}^{(\0)}(\T\M) \to D^q \x \M\)).  Write \(\cal{R} \x_{(D^q \x \M)} \Gr_{n-1}^{(\infty)}(\T\M_{D^q})\) for the bundle over \(D^q \x \M\) given by taking the fibrewise product of \(\cal{R}\) and \(\Gr_{n-1}^{(\infty)}(\T\M_{D^q})\); explicitly:
\ew
\cal{R} \x_{(D^q \x \M)} \Gr_{n-1}^{(\infty)}(\T\M_{D^q}) &=\\
& \hs{-33mm} \lt\{[(s,T),(s,\Xi)] \in \cal{R} \x \Gr_{n-1}^{(\infty)}(\T\M_{D^q}) \cc \lt(D^q \x E^{(1)}\rt) \x \lt(D^q \x \Gr_{n-1}^{(\infty)}(\T\M)\rt) ~\m|~ \pi_{E^{(1)}}(T) = \pi_{\Gr_{n-1}^{(\infty)}(\T\M)}(\Xi)\rt\},
\eew
where \(\pi_{E^{(1)}}\) and \(\pi_{\Gr_{n-1}^{(\infty)}(\T\M)}\) denote the bundle projections \(E^{(1)} \to \M\) and \(\Gr_{n-1}^{(\infty)}(\T\M) \to \M\) respectively.  Let \(\cal{A} \cc \cal{R} \x_{(D^q \x \M)} \Gr_{n-1}^{(\infty)}(\T\M_{D^q})\). Given \(s \in D^q\), \(x\in\M\) and a configuration of hyperplanes \((s,\Xi) \in \Gr_{n-1}^{(\0)}(\T\M_{D^q})_{(s,x)} = \{s\} \x \Gr_{n-1}^{(\infty)}(\T_x\M)\), there is a natural subset \(\cal{A}(s,\Xi) \cc E^{(1)}_x\) given by:
\ew
\cal{A}(s,\Xi) = \{ T \in E^{(1)}_x ~|~ [(s,T),(s,\Xi)] \in \cal{A}_{(s,x)}\}.
\eew
Similarly, given a 1-jet \((s,T) \in \cal{R}_{(s,x)}\), there is a natural subset \(\cal{A}(s,T) \cc \Gr_{n-1}^{(\infty)}(\T_x\M)\) given by:
\ew
\cal{A}(s,T) = \{ \Xi \in \Gr_{n-1}^{(\infty)}(\T_x\M) ~|~ [(s,T),(s,\Xi)] \in \cal{A}_{(s,x)} \}.
\eew

\begin{Defn}[{Cf.\ \cite[Defn.\ 4.1]{CIwA&H46D}}]\label{CIwA-defn}
Let \(\M\), \(q\) and \(\cal{R}\) be as above.  Say \(\cal{A} \cc \cal{R} \x_{D^q \x \M} \Gr_{n-1}^{(\infty)}(\T\M_{D^q})\) is a fibred avoidance pre-template for \(\cal{R}\) if:
\begin{enumerate}
\item \(\cal {A} \cc \cal{R} \x_{(D^q \x \M)} \Gr_{n-1}^{(\infty)}(\T\M_{D^q})\) is an open subset;
\item For all \(s \in D^q\), \(x \in \M\) and all pairs \(\Xi' \cc \Xi \in \Gr_{n-1}^{(\infty)}(\T_x\M)\), there is an inclusion \(\cal{A}(s,\Xi) \cc \cal{A}(s,\Xi')\).
\end{enumerate}
Say that \(\cal{A}\) is a fibred avoidance template for \(\cal{R}\) if it also satisfies the following two conditions:
\begin{enumerate}
\setcounter{enumi}{2}
\item For all \(s \in D^q\), \(x \in \M\) and \((s,T) \in \cal{R}_{(s,x)}\), the subset \(\cal{A}(s,T) \cc \Gr_{n-1}^{(\infty)}(\T_x\M)\) is dense (and open);
\item For all \(s \in D^q\), \(x \in \M\), \(\Xi \in \Gr_{n-1}^{(\infty)}(\T_x\M)\), \(\bb{B} \in \Xi\), \(\la \in \Hom(\bb{B},E_x)\) and \(e \in E_x\), the subset \(\cal{A}(s,\Xi) \cap \Pi_e(\bb{B},\la) \cc \Pi_e(\bb{B},\la)\) is ample, meaning that either \(\cal{A}(s,\Xi) \cap \Pi_e(\bb{B},\la) = \es\), or every path component of \(\cal{A}(s,\Xi) \cap \Pi_e(\bb{B},\la)\) has convex hull equal to \(\Pi_e(\bb{B},\la)\).
\end{enumerate}
\end{Defn}

\begin{Thm}[{\cite[Thm.\ 5.1; see also Lem.\ 4.7]{CIwA&H46D}}]\label{CIwA-Thm}
Let \(\M\) be an \(n\)-manifold, let \(E \to \M\) be a vector bundle, let \(q \ge 0\) and let \(\cal{R}\cc E^{(1)}_{D^q}\) be an open fibred differential relation on sections of \(E\). Suppose that \(\cal{R}\) admits a fibred avoidance template \(\cal{A} \cc \cal{R} \x_{(D^q \x \M)} \Gr_{n-1}^{(\infty)}(\T\M_{D^q})\). Then, \(\cal{R}\) satisfies the relative \(h\)-principle.
\end{Thm}

\tref{CIwA-Thm} is a special case of Gromov's general theory of convex integration via convex hull extensions introduced in \cite{PDR} and developed in \cite{CIT} (see \cite[Cor.\ 5.5]{CIwA&H46D}).  Note also that \(\cal{A} = \cal{R} \x_{(D^q \x \M)} \Gr_{n-1}^{(\infty)}(\T\M_{D^q})\) is a fibred avoidance template for \(\cal{R}\) \iff\ \(\cal{R}\) is an ample fibred relation in the classical sense and thus, in this case, \tref{CIwA-Thm} recovers the classical convex integration theorem as proved in \cite[Chs.\ 17--18]{ItthP}.

\begin{Rk}\label{calc-with-amp}
The fibred avoidance pre-templates considered in this paper will all be of the form:
\ew
\cal{A} = E_{D^q} \x_{(D^q \x \M)} \cal{A}' \cc E_{D^q} \x_{(D^q \x \M)} \bigg[\lt(\T^*\M \ts E\rt)_{D^q} \x_{(D^q \x \M)} \Gr_{n-1}^{(\0)}(\T\M_{D^q})\bigg]
\eew
for some subbundle \(\cal{A}' \cc \lt(\T^*\M \ts E\rt)_{D^q} \x_{(D^q \x \M)} \Gr_{n-1}^{(\0)}(\T\M_{D^q})\).  In this case, given \(s \in D^q\), \(x\in\M\) and \(\Xi \in \Gr_{n-1}^{(\infty)}(\T_x\M)\), define
\ew
\cal{A}'(s,\Xi) = \lt\{ T \in \T^*_x\M \ts E_x ~\m|~ [(s,T),(s,\Xi)] \in \cal{A}'_{(s,x)}\rt\}.
\eew
Then, for all \(\bb{B} \in \Xi\), \(\la \in \Hom(\bb{B},E_x)\) and \(e \in E_x\):
\ew
\cal{A}(s,\Xi) \cap \Pi_e(\bb{B},\la) = \{e\} \x \Big[\cal{A}(s,\Xi)' \cap \Pi(\bb{B},\la)\Big]
\eew
for \(\Pi(\bb{B},\la)\) as defined in \eref{PiBl-defn}, and thus \(\cal{A}(s,\Xi) \cap \Pi_e(\bb{B},\la) \cc \Pi_e(\bb{B},\la)\) is ample \iff\ \(\cal{A}'(s,\Xi) \cap \Pi(\bb{B},\la) \pc \Pi(\bb{B},\la)\) is ample for all \(\bb{B}\) and \(\la\).
\end{Rk}

\begin{Rk}[Cohomology of \(\Op(A)\)]
Given a polyhedron \(A\) in a manifold \(\M\), note that every sufficiently small open neighbourhood \(U\) of \(A\) deformation retracts onto \(A\).  In particular, one can always implicitly assume that \(\Op(A)\) has been chosen small enough that \(A\) and \(\Op(A)\) have identical cohomology rings and thus condition (2) in the introduction is independent of the choice of \(\Op(A)\).\\
\end{Rk}

\section{Formulating the \(h\)-principle for \slr\ 3-forms as a differential relation}\label{DR-form}

Let \(\M\) be an oriented \(6\)-manifold and recall that the symbol of the exterior derivative on 2-forms is the unique vector bundle homomorphism \(\mc{D}:\ww{2}\T^*\M^{(1)}\to\ww{3}\T^*\M\) such that the following diagram commutes:
\ew
\begin{tikzcd}
\Ga\lt(\M,\ww{2}\T^*\M^{(1)}\rt) \ar[rr, "\mLa{\mc{D}}"] & & \Om^{3}(\M)\\
& \Om^{2}(\M) \ar[ru, "\mLa{\dd}"'] \ar[lu, "\mLa{j_1}"] &
\end{tikzcd}
\eew
where \(\ww{2}\T^*\M^{(1)}\) denotes the first jet bundle of \(\ww{2}\T^*\M\).  Explicitly, identifying \(\ww{2}\T^*\M^{(1)} \cong \ww{2}\T^*\M \ds \lt(\T^*\M \ts \ww{2}\T^*\M\rt)\) as usual, \(\mc{D}\) is simply the composite map:
\ew
\ww{2}\T^*\M \ds \lt(\T^*\M \ts \ww{2}\T^*\M\rt) \oto{\mns{~proj_2~}} \T^*\M \ts \ww{2}\T^*\M \oto{\mns{\w}} \ww{3}\T^*\M.
\eew
Now, fix \(q \ge 0\), let \(a: D^q \to \Om^3(\M)\) be any continuous map and define a fibred differential relation \(\cal{R}_+(a) \cc D^q \x \ww{2}\T^*\M^{(1)}\) by:
\ew
\cal{R}_+(a) =&\ \lt\{ (s,T) \in D^q \x \ww{2}\T^*\M^{(1)} ~\m|~ \mc{D}(T) + a(s) \in \ww[+]{3}\T^*\M \rt\}\\
=&\ \mc{D}^{-1}\lt(\ww[+]{3}\T^*\M_{D^q} - a\rt).
\eew
As proven in \cite[Lem.\ 5.2]{RhPfCSF}, if the fibred differential relation \(\cal{R}_+(a)\) satisfies the relative \(h\)-principle for all \(a\), then \slr\ 3-forms satisfy the relative \(h\)-principle.

I begin by remarking that, unlike the examples considered in \cite{RhPfCSF}, \(\cal{R}_+(a) \x_{(D^q \x \M)} \Gr_5^{(\0)}(\T\M_{D^q})\) itself is not a fibred avoidance template for \(\cal{R}_+(a)\).  Indeed, by \cite[Prop.\ 5.4]{RhPfCSF}, \(\cal{R}_+(a) \x_{D^q \x \M} \Gr_5^{(\0)}(\T\M_{D^q})\) is a fibred avoidance template for \(\cal{R}_+(a)\) \iff\ \(\mc{N}_{\rh_+}(\ta) \pc \ww{2}\lt(\bb{R}^5\rt)^*\) is ample for every \(\ta \in \ww{3}\lt(\bb{R}^5\rt)^*\).  However \(\mc{N}_{\rh_+}(\ta) \pc \ww{2}\lt(\bb{R}^5\rt)^*\) need not be ample.  To see this, consider the standard \slr\ 3-form \(\rh_+ = e^{123} + e^{456}\) on \(\bb{R}^6\) and recall the \(\pm1\)-eigenspaces of the para-complex structure \(I_{\rh_+}\):
\ew
E_+ = \<e_1,e_2,e_3\? \et E_-=\<e_4,e_5,e_6\?.
\eew
Given a hyperplane \(\bb{B}\pc\bb{R}^6\), on dimensional grounds one of the following statements holds:
\begin{enumerate}
\item \(\dim(\bb{B} \cap E_\pm) = 2\);
\item \(\dim(\bb{B} \cap E_+) = 2\) but \(\dim(\bb{B} \cap E_-) = 3\) (equivalently \(E_- \pc \bb{B}\));
\item \(\dim(\bb{B} \cap E_-) = 2\) but \(\dim(\bb{B} \cap E_+) = 3\) (equivalently \(E_+ \pc \bb{B}\)).
\end{enumerate}
Denote the sets of oriented hyperplanes corresponding to (1), (2) and (3) above by \(\oGr_{5,gen}\lt(\bb{R}^6\rt)\), \(\oGr_{5,-}\lt(\bb{R}^6\rt)\) and \(\oGr_{5,+}\lt(\bb{R}^6\rt)\) respectively.
\begin{Prop}
\ew
\lqt{\oGr_5\lt(\bb{R}^6\rt)}{\SL(3;\bb{R})^2} = \lt\{\oGr_{5,gen}\lt(\bb{R}^6\rt), \oGr_{5,-}\lt(\bb{R}^6\rt), \oGr_{5,+}\lt(\bb{R}^6\rt)\rt\}.
\eew
\end{Prop}
\begin{proof}
Firstly note that there is an isomorphism:
\ew
\bcd[row sep = 0pt]
\oGr_{5,+}\lt(\bb{R}^6\rt) \ar[r] & \oGr_2\lt(E_-\rt)\\
\Pi \ar[r, maps to] & \Pi \cap E_-\hs{1.5pt},
\ecd
\eew
where \(\Pi \cap E_-\) is oriented via the decomposition \(\Pi = E_+ \ds \lt(\Pi \cap E_-\rt)\).  Recalling that \(\SL(3;\bb{R})^2\) acts on \(\bb{R}^6\) diagonally via the decomposition \(\bb{R}^6 = E_+ \ds E_-\), and that \({ \bf1} \x \SL(3;\bb{R})\) acts transitively on \(\oGr_2\lt(E_-\rt)\), it follows that \(\oGr_{5,+}\lt(\bb{R}^6\rt)\) is a single orbit for the action of \(\SL(3;\bb{R})^2\).  Likewise, \(\oGr_{5,-}\lt(\bb{R}^6\rt)\) is a single orbit.

In the remaining case, firstly note that \(\Gr_{5,gen}\lt(\bb{R}^6\rt)\) forms a single orbit for \(\SL(3;\bb{R})^2\).  Indeed, there is a natural line bundle \(\mc{L}_+\) over \(\Gr_2(E_+)\) with fibre over \(\pi_+ \in \Gr_2(E_+)\) given by:
\ew
\mc{L}_+|_{\pi_+} = \rqt{E_+}{\pi_+}.
\eew
The action of \(\SL(3;\bb{R}) \x { \bf1}\) on \(\Gr_2(E_+)\) lifts naturally to define an action on \(\mc{L}_+\) which can be shown to act transitively on \(\mc{L}_+\osr\Gr_2(E_+)\), the complement of the zero section.  The analogous statement holds for \(\mc{L}_-\osr\Gr_2(E_-)\).  Now, note that there is a surjective map:
\ew
\bcd[row sep = 0pt]
\mc{L}_+\osr\Gr_2(E_+) \x \mc{L}_-\osr\Gr_2(E_-) \ar[r] & \Gr_{5,gen}\lt(\bb{R}^6\rt)\\
\lt(u_+ + \pi_+ \in \rqt{E_+}{\pi_+}, u_- + \pi_- \in \rqt{E_-}{\pi_-}\rt) \ar[r, maps to] & \pi_+ \ds \pi_- \ds \<u_+ + u_-\?.
\ecd
\eew
Since \(\SL(3;\bb{R})^2\) acts transitively on \(\mc{L}_+\osr\Gr_2(E_+) \x \mc{L}_-\osr\Gr_2(E_-)\), it follows that \(\Gr_{5,gen}\lt(\bb{R}^6\rt)\) forms a single \(\SL(3;\bb{R})^2\)-orbit, as claimed.  Therefore, to verify that \(\oGr_{5,gen}\lt(\bb{R}^6\rt)\) forms a single orbit, it suffices to consider \(\bb{B} \in \oGr_{5,gen}\lt(\bb{R}^6\rt)\) with oriented basis \(\<e_1,e_2,e_4,e_5, e_3 + e_6\?\) and note that:
\ew
F = \bpm -1 & & & & & \\ & 1 & & & & \\ & & -1 & & & \\ & & & -1 & & \\ & & & & 1 & \\ & & & & & -1 \epm \in \SL(3;\bb{R})^2
\eew
preserves \(\bb{B}\) and \(F|_\bb{B}\) is orientation-reversing.

\end{proof}

Clearly \(\oGr_{5,gen}\lt(\bb{R}^6\rt) \pc \oGr_5\lt(\bb{R}^6\rt)\) is open and dense. By \pref{open-to-open}, it follows that \(\mc{T}_{\rh_+}\lt(\oGr_{5,gen}\lt(\bb{R}^6\rt)\rt)\) must be the (unique) open orbit of 3-forms on \(\bb{R}^5\), i.e.\ the orbit of the 3-form \(\th^{123} + \th^{145}\).  Denote this orbit \(\ww[OP]{3}\lt(\bb{R}^5\rt)\) and term forms in this orbit ospseudoplectic, in the terminology of \cite[Prop.\ 3.12]{RhPfCSF}.  Now, consider the orbit \(\oGr_{5,+}\lt(\bb{R}^6\rt)\).  Taking \(\bb{B} = \<e_1,...,e_5\? \in \oGr_{5,+}\lt(\bb{R}^6\rt)\) yields:
\ew
\rh_+|_\bb{B} = \th^{123}.
\eew
It follows that \(\mc{T}_{\rh_+}\lt(\oGr_{5,+}\lt(\bb{R}^6\rt)\rt)\) is the orbit of non-zero, decomposable 3-forms on \(\bb{R}^5\).  By considering \(\bb{B} = \<e_2,...,e_6\? \in \oGr_{5,-}\lt(\bb{R}^6\rt)\), one sees that \(\mc{T}_{\rh_+}\lt(\oGr_{5,-}\lt(\bb{R}^6\rt)\rt)\) is precisely the same orbit.

\begin{Prop}\label{slr-setup}
Let \(\ta \in \ww[OP]{3}\lt(\bb{R}^5\rt)\).  Then, \(\mc{N}_{\rh_+}(\ta)\) is ample.  In contrast, now let \(\ta\) be a non-zero decomposable 3-form on \(\bb{R}^5\).  Then, \(\mc{N}_{\rh_+}(\ta)\) consists of two convex, connected components; in particular, it is not ample.
\end{Prop}

\begin{proof}
Let \(\ta \in \ww[OP]{3}\lt(\bb{R}^5\rt)^*\).  Then, \(\Stab_{\GL_+(5;\bb{R})}(\ta)\) is connected by \cite[Prop.\ 3.14]{RhPfCSF} and:
\ew
\mc{T}_{\rh_+}^{-1}\lt(\lt\{\mc{T}_{\rh_+}\lt[\oGr_{5,gen}\lt(\bb{R}^6\rt)\rt]\rt\}\rt) = \lt\{\oGr_{5,gen}\lt(\bb{R}^6\rt)\rt\},
\eew
by the above discussion.  Since \(\rh_+\) admits the orientation-reversing automorphism:
\ew
e_1 \leftrightarrow e_4, \hs{3mm} e_2 \leftrightarrow e_5, \hs{3mm} e_3 \leftrightarrow e_6
\eew
it follows from \lref{ab-lem} that \(\mc{N}_{\rh_+}(\ta)\) is ample.

Now let \(\ta\) be a non-zero, decomposable 3-form.  Identify \(\bb{R}^5\) with the subspace \(\<e_2,...,e_6\?\) of \(\bb{R}^6\) and take \(\ta = \th^{456}\). Then:
\ew
\mc{N}_{\rh_+}(\ta) = \lt\{\om \in \ww{2}\<\th^2,...,\th^6\?~\middle|~\th^1 \w \om + \th^{456} \in \ww[+]{3}\lt(\bb{R}^6\rt)^*\rt\}.
\eew
Recall that a 3-form \(\rh \in \ww{3}\lt(\bb{R}^6\rt)^*\) is of \(\SL(3;\bb{R})^2\)-type \iff\ the quadratic invariant \(\La\) defined in \sref{pre:slr} is positive.  A direct calculation shows that:
\ew
\La(\th^1 \w \om + \th^{456}) = \om(e_2,e_3)^2\cdot(\th^{123456})^{\ts2}.
\eew
Thus:
\ew
\mc{N}_{\rh_+}(\ta) = \lt\{\om \in \ww{2}\<\th^2,...,\th^6\?~\middle|~\om(e_2,e_3) \ne 0\rt\},
\eew
which has the form claimed.

\end{proof}

\section{Defining a fibred avoidance template for \(\cal{R}_+(a)\)}\label{Defn-AT}

The aim of this section is to define a fibred avoidance template \(\cal{A}\) for \(\cal{R}_+(a)\) and prove that it satisfies conditions (1)--(3) in \dref{CIwA-defn}.

\begin{Defn}
Let \(\rh\in\ww[+]{3}\lt(\bb{R}^6\rt)^*\) be an \slr\ 3-form and let \(\{\bb{B}_1,...,\bb{B}_k\}\in\Gr_5^{(k)}\lt(\bb{R}^6\rt)\) be a configuration of hyperplanes in \(\bb{R}^6\). Say that \(\{\bb{B}_1,...,\bb{B}_k\}\) is generic \wrt\ \(\rh\) if \(\bb{B}_i \in \Gr_{5,gen}\lt(\bb{R}^6\rt)\) for all \(i\in\{1,...,k\}\), and if for all distinct \(i,j\in\{1,...,k\}\) at least one of the conditions:
\ew
\bb{B}_i \cap E_{+,\rh} \ne \bb{B}_j \cap E_{+,\rh} \abo \bb{B}_i \cap E_{-,\rh} \ne \bb{B}_j \cap E_{-,\rh}
\eew
holds.  Write \(\Gr_{5,gen}^{(\infty)}\lt(\bb{R}^6\rt)_\rh\) for the collection of all generic configurations of hyperplanes in \(\bb{R}^6\) \wrt\ \(\rh\), or simply \(\Gr_{5,gen}^{(\infty)}\lt(\bb{R}^6\rt)\), when \(\rh\) is clear from context.  Note that, formally, \(\Gr_{5,gen}^{(1)}\lt(\bb{R}^6\rt) = \Gr_{5,gen}\lt(\bb{R}^6\rt)\); note also that for \(k \ge 2\), \(\Xi = \{\bb{B}_1,...,\bb{B}_k\}\) is generic \iff\ every subset of \(\Xi\) of size 2 is generic.
\end{Defn}

The appellation `generic' is justified by the following proposition:
\begin{Prop}\label{o&d-prop}
Let \(\rh\in\ww[+]{3}\lt(\bb{R}^6\rt)^*\) be an \slr\ 3-form. Then, \(\Gr_{5,gen}^{(\infty)}\lt(\bb{R}^6\rt) \pc \Gr_5^{(\infty)}\lt(\bb{R}^6\rt)\) is an open and dense subset.
\end{Prop}

\begin{proof}
Recall from above that \(\Gr_{5,gen}\lt(\bb{R}^6\rt) \pc \Gr_5\lt(\bb{R}^6\rt)\) is open and dense.  Thus, it suffices to prove that \(\Gr_{5,gen}^{(k)}\lt(\bb{R}^6\rt) \pc \Gr_5^{(k)}\lt(\bb{R}^6\rt)\) is open and dense for every \(k \ge 2\).

Fix \(k \ge 2\) and recall that \(\Gr_5^{(k)}\lt(\bb{R}^6\rt)\) may be identified with the quotient \(\rqt{\tld{\prod_1^k \Gr_5\lt(\bb{R}^6\rt)}}{Sym_k}\), where:
\ew
\tld{\mns{\prod_1^k}\Gr_5\lt(\bb{R}^6\rt)} = \lt\{\lt(\bb{B}_1,...,\bb{B}_k\rt) \in \prod_1^k\Gr_5\lt(\bb{R}^6\rt)~\m|~ \bb{B}_i \ne \bb{B}_j \text{ for all \(i\ne j\)}\rt\}.
\eew
Define \(\mc{G} \pc \tld{\prod_1^k\Gr_5\lt(\bb{R}^6\rt)}\) to be the preimage of \(\Gr_{5,gen}^{(k)}\lt(\bb{R}^6\rt)\) under the quotient map \(\si:\tld{\prod_1^k \Gr_5\lt(\bb{R}^6\rt)} \to \rqt{\tld{\prod_1^k \Gr_5\lt(\bb{R}^6\rt)}}{Sym_k} \cong \Gr_5^{(k)}\lt(\bb{R}^6\rt)\); explicitly:
\ew
\mc{G} = \lt\{\lt(\bb{B}_1,...,\bb{B}_k\rt) \in \prod_1^k\Gr_{5,gen}\lt(\bb{R}^6\rt)~\m|~\text{for all \(i\ne j\), } \bb{B}_i\cap E_+ \ne \bb{B}_j \cap E_+ \text{ or } \bb{B}_i\cap E_- \ne \bb{B}_j \cap E_-\rt\}.
\eew

Since \(\si\) is open and surjective, to prove the proposition it suffices to prove that \(\mc{G} \pc \tld{\prod_1^k\Gr_5\lt(\bb{R}^6\rt)}\) is open and dense, or equivalently that \(\mc{G} \pc \prod_1^k \Gr_{5,gen}\lt(\bb{R}^6\rt)\) is open and dense (since both \(\prod_1^k \Gr_{5,gen}\lt(\bb{R}^6\rt)\) and \(\tld{\prod_1^k \Gr_5\lt(\bb{R}^6\rt)}\) are themselves open and dense subsets of \(\prod_1^k\Gr_5\lt(\bb{R}^6\rt)\)).

To this end, note that there is an inclusion:
\e\label{mcS}
\lt.\prod_1^k \Gr_{5,gen}\lt(\bb{R}^6\rt)\rt\osr\mc{G} \pc \lt\{\lt(\bb{B}_1,...,\bb{B}_k\rt) \in \prod_1^k\Gr_{5,gen}\lt(\bb{R}^6\rt)~\middle|~\text{for some \(i\ne j\): } \bb{B}_i\cap E_+ = \bb{B}_j \cap E_+\rt\} = \mc{S}.
\ee
\(\mc{S}\) is a stratified submanifold of \(\prod_1^k \Gr_{5,gen}\lt(\bb{R}^6\rt)\) of codimension 2.  Indeed, there is an \(\SL(3;\bb{R})^2\)-equivariant map:
\ew
\bcd[row sep = 0pt]
\cap^+:\Gr_{5,gen}\lt(\bb{R}^6\rt) \ar[r] & \Gr_2(E_+)\\
\bb{B} \ar[r, maps to] & \bb{B} \cap E_+
\ecd
\eew
which is submersive since \(\SL(3;\bb{R})^2\) acts transitively on \(\Gr_2(E_+)\).  Taking the Cartesian product yields a submersion:
\ew
\prod_1^k \cap^+: \prod_1^k \Gr_{5,gen}\lt(\bb{R}^6\rt) \to \prod_1^k \Gr_2(E_+).
\eew
By definition:
\ew
\mc{S} = \lt(\prod_1^k \cap^+\rt)^{-1} \lt(\prod_1^k \Gr_2(E_+)\rt)_{sing}
\eew
(see \sref{conf-intro}) where the set \(\lt(\prod_1^k \Gr_2(E_+)\rt)_{sing} \pc \prod_1^k \Gr_2(E_+)\) is a stratified submanifold of codimension \(\dim\Gr_2(E_+) = 2\). Using the Preimage Theorem (which applies equally well to stratified submanifolds; see e.g.\ \cite[p.\ 17]{ItthP}), it follows that \(\mc{S}\) is a stratified submanifold of codimension 2. The openness and density of \(\mc{G}\) in \(\prod_1^k \Gr_{5,gen}\lt(\bb{R}^6\rt)\) now follows from \eref{mcS}, completing the proof.

\end{proof}

\begin{Defn}
Let \(\M\) be an oriented 6-manifold, fix \(q \ge 0\) and let \(a: D^q \to \Om^3(\M)\) be a continuous map. Define:
\ew
\cal{A} = \lt\{[(s,T),(s,\Xi)] \in \cal{R}_+(a) \x_{(D^q \x \M)} \Gr_5^{(\infty)}(\T\M_{D^q})~\m|~\Xi \in \Gr_{5,gen}^{(\infty)}(\T\M)_{\mc{D}(T) + a(s)}\rt\}.
\eew 
\end{Defn}

\begin{Prop}
\(\cal{A}\) is a pre-template for \(\cal{R}_+(a)\). Moreover, for each \(s \in D^q\), \(x \in \M\) and \((s,T) \in \cal{R}_+(a)_{(s,x)}\):
\ew
\cal{A}(s,T) \pc \Gr_5^{(\infty)}(\T_x\M)
\eew
is a(n) (open and) dense subset.
\end{Prop}

\begin{proof}
It is clear that \(\cal{A} \pc \cal{R}_+(a) \x_{(D^q \x \M)} \Gr_5^{(\infty)}(\T\M_{D^q})\) is open, since for \(\rh \in \ww[+]{3}\lt(\bb{R}^6\rt)^*\) and \(\Xi \in \Gr_5^{(\infty)}\lt(\bb{R}^6\rt)\), the condition \(\Xi \in \Gr_{5,gen}^{(\infty)}\lt(\bb{R}^6\rt)_\rh\) is open in both \(\rh\) and \(\Xi\).  Now, fix \(s \in D^q\) and \(x \in \M\), consider \(\Xi' \cc \Xi \in \Gr_5^{(\infty)}(\T_x\M)\) and suppose \(T \in \cal{A}(s,\Xi) \cc E^{(1)}_x\). Write \(\rh = \mc{D}(T) + a(s)\). Then, \(\Xi \in \Gr_5^{(\infty)}(\T_x\M)_{\rh,gen}\) and so, since \(\Xi' \cc \Xi\), it follows that \(\Xi'\in \Gr_{5,gen}^{(\infty)}(\T_x\M)_\rh\) and hence that \(T \in \cal{A}(s,\Xi')\). Thus, \(\cal{A}(s,\Xi) \cc \cal{A}(s,\Xi')\) and hence \(\cal{A}\) is a pre-template for \(\cal{R}_+(a)\), as claimed. The final claim follows immediately from \pref{o&d-prop}.

\end{proof}

Note that the pre-template \(\cal{A}\) has the form described in \rref{calc-with-amp}.  Thus, to prove that \(\cal{A}\) is a fibred avoidance template for \(\cal{R}_+(a)\), and hence complete the proof of \tref{slr-hP-thm}, it suffices to prove that for all \(s \in D^q\), \(x \in \M\), \(\Xi \in \Gr_5^{(\infty)}(\T_x\M)\), \(\bb{B} \in \Xi\) and \(\la \in \Hom(\bb{B},\ww{2}\T^*_x\M)\), the subset:
\ew
\cal{A}'(s,\Xi) \cap \Pi(\bb{B},\la) \cc \Pi(\bb{B},\la)
\eew
is ample.  Fix \(\bb{B} \in \Xi\), choose an orientation on \(\bb{B}\), fix an oriented splitting \(\T_x\M = \bb{L} \ds \bb{B}\) and choose an oriented generator \(\th\) of the 1-dimensional oriented vector space \(\Ann(\bb{B}) \pc \T^*_x\M\).  Then, there is an isomorphism:
\ew
\bcd[row sep = 0pt]
\bb{B}^* \ds \ww{2}\bb{B}^* \ds \lt(\bb{B}^* \ts \ww{2}\T^*_x\M\rt) \ar[r] & \T^*_x\M \ts \ww{2}\T^*_x\M \ar[l]\\
\al \ds \nu \ds \la \ar[r, maps to] & \th \ts \lt(\th \w \al + \nu\rt) + \la.
\ecd
\eew
Using this identification:
\ew
\Pi(\bb{B},\la) \cong \bb{B}^* \x \ww{2}\bb{B}^* \x \{\la\}
\eew
and thus:
\ew
\cal{A}'(s,\Xi) \cap \Pi(\bb{B},\la) \cong \bb{B}^* \x \lt\{\nu \in \ww{2}\bb{B}^* ~\m|~ \raisebox{1pt}{\parbox{55mm}{\(\th \w \nu + \w(\la) + a(s)|_x \in \ww[+]{3}\T^*_x\M\) and \(\Xi\) is generic for \(\th \w \nu + \w(\la) + a(s)|_x\)}} \rt\} \x \{\la\}.
\eew

In particular, the ampleness of \(\cal{A}'(s,\Xi) \cap \Pi(\bb{B},\la)\) depends only on \(\w(\la)\) (for a fixed choice of \(a\)).  Thus, writing \(\ta = \w(\la) + a(s)|_x\), the task is to prove that for each \(\ta \in \ww{3}\bb{B}^*\), the subset:
\ew
\mc{N}(\ta;\Xi,\bb{B}) = \lt\{\nu \in \ww{2}\bb{B}^* ~\m|~ \th \w \nu + \ta \in \ww[+]{3}\T^*_p\M \text{ and } \Xi \text{ is generic for } \th \w \nu + \ta \rt\} \pc \ww{2}\bb{B}^*
\eew
is ample.  If this set is empty, the result is trivial, so \wlg\ one may assume that there exists \(\nu_0 \in \ww{2}\bb{B}^*\) such that \(\rh = \th \w \nu_0 + \ta\) is an \slr\ 3-form on \(\T_p\M\) \wrt\ which \(\Xi\) is generic.  Since \(\mc{N}(\ta;\Xi,\bb{B}) = \mc{N}(\rh;\Xi,\bb{B}) + \nu_0\), one sees that to prove \tref{slr-hP-thm}, it suffices to prove:
\begin{Prop}\label{A-Ample}
Let \(\rh \in \ww[+]{3}\lt(\bb{R}^6\rt)\) be an \slr\ 3-form, let \(\Xi \in \Gr^{(\infty)}_5\lt(\bb{R}^6\rt)\) be a generic configuration of hyperplanes \wrt\ \(\rh\), let \(\bb{B} \in \Xi\), choose an orientation on \(\bb{B}\), fix an oriented splitting \(\bb{R}^6 = \bb{L} \ds \bb{B}\) and choose an oriented generator \(\th\) of the 1-dimensional oriented vector space \(\Ann(\bb{B}) \pc \lt(\bb{R}^6\rt)^*\).  Define:
\ew
\mc{N}(\rh;\Xi,\bb{B}) = \lt\{\nu \in \ww{2}\bb{B}^* ~\m|~ \th \w \nu + \rh \in \ww[+]{3}\lt(\bb{R}^6\rt)^* \text{and } \Xi \text{ is generic for } \th \w \nu + \rh \rt\}.
\eew
Then, \(\mc{N}(\rh;\Xi,\bb{B}) \pc \ww{2}\bb{B}^*\) is ample.
\end{Prop}

I begin with a lemma:
\begin{Lem}\label{conn-lem}
Let \(X\) be a connected topological space and let \(Y \cc X\) have empty interior.  Suppose that for every \(y \in Y\), there exists an open neighbourhood \(U_y\) of \(y\) in \(X\) such that \(U_y \osr Y\) is connected.  Then, \(X \osr Y\) is connected.
\end{Lem}

\begin{proof}
The proof is a simple exercise in point-set topology.  Suppose that \(A,B \cc X \osr Y\) are open, disjoint subsets such that \(X \osr Y = A \cup B\).  For each \(y \in Y\), since \(U_y \osr Y\) is connected, it follows that either:
\e\label{dichot}
U_y \osr Y \cc A \abo U_y \osr Y \cc B.
\ee
Thus, define:
\e\label{defining-extension}
A' = A \cup \bigg\{y \in Y ~\bigg|~ \raisebox{3.3pt}{\parbox{6cm}{\center{there exists some open neighbourhood \(W_y\) of \(y\) in \(X\) such that \(W_y \osr Y \cc A\)}}} \bigg\}
\ee
and define \(B'\) analogously.  Then, by \eref{dichot}, clearly \(A' \cup B' = A \cup B \cup Y = X\).  Next, note that \(A' \cc X\) is open.  Indeed, since \(A \pc X \osr Y\) is open, there exists an open subset \(\mc{O} \cc X\) such that \(A = \mc{O} \cap (X \osr Y)\).  Then, every \(y \in \mc{O} \cap Y\) also lies in \(A'\) (simply take \(W_y = \mc{O}\)) so \(A \cc \mc{O} \cc A'\).  Now, let \(y \in Y \cap A'\) and let \(W_y\) be as in \eref{defining-extension}.  Then, every \(y' \in W_y \cap Y\) also lies in \(A'\) (simply take \(W_{y'} = W_y\)) and so \(y \in W_y \cc A'\).  Thus:
\ew
A' \cc \mc{O} \cup \bigcup_{y \in Y \cap A'} W_y \cc A',
\eew
hence equality holds, and whence \(A'\) is open.  Similarly, \(B' \cc X\) is also open.

Now, suppose there exists \(y \in A' \cap B'\).  Then, clearly \(y \in Y\) (since \(A' \cap B' \cap (X \osr Y) = A \cap B = \es\)).  By definition, there exist neighbourhoods \(W_y\) and \(W_y'\) of \(y\) in \(X\) such that \(W_y \osr Y \cc A\) and \(W_y' \osr Y \cc B\).  Then:
\ew
(W_y \cap W_y') \cap (X \osr Y) \cc A \cap B = \es,
\eew
which contradicts the density of \(X \osr Y\) (since \(W_y \cap W_y'\) is an open neighbourhood of \(y\) in \(X\)).  Thus, \(A' \cap B' = \es\).  Since \(X\) is connected, it follows that one of \(A'\) and \(B'\) must be empty, and hence so must one of \(A\) and \(B\).

\end{proof}

Now let \(\bb{A}\) be an affine space and \(X \cc \bb{A}\) an open subset.  I term a subset \(Y \cc X\) macilent if it is closed and if, for every point \(y \in Y\), there exists an open neighbourhood \(U_y\) of \(y\) in \(X\) and a submanifold \(S_y \pc U_y\) of codimension at least 2 such that:
\e\label{loc-small}
Y \cap U_y \cc S_y.
\ee
\begin{Lem}\label{ample-vs-mac}
Let \(X \cc \bb{A}\) be open and ample.  If \(Y \cc X\) is macilent, then \(X \osr Y\) is also open and ample.
\end{Lem}

\begin{Rk}
A related result concerning so-called `thin' sets was stated without proof in \cite[\S18.1]{ItthP} however, to the author's knowledge, the notion of macilent sets used in this paper cannot be found in the literature.
\end{Rk}

\begin{proof}
By considering each path component of \(X\) separately, it suffices to consider the case where \(X\) is open, path-connected and ample (i.e.\ satisfies \(\Conv(X) = \bb{A}\)).  Since each \(S_y\) has codimension at least 2 in \(U_y\), it follows that \(Y\) has empty interior in \(X\) and that \(U_y \osr S_y\) is connected for all \(y \in Y\).  But \(U_y \osr S_y\) is dense in \(U_y\), hence certainly dense in \(U_y \osr Y\) and whence \(U_y \osr Y\) is also connected for all \(y \in Y\).  It follows from \lref{conn-lem} that \(X \osr Y\) is connected.  Since \(X \osr Y\) is open in \(X\) and \(X\) is open in \(\bb{A}\), it follows that \(X \osr Y\) is also locally path-connected and hence path-connected, as claimed.  To see that \(\Conv(X \osr Y) = \bb{A}\), note that for each \(y \in Y\), by \eref{loc-small}:
\ew
y \in \Conv(U_y \osr Y) \cc \Conv(X \osr Y)
\eew
and hence:
\ew
\Conv(X \osr Y) = \Conv(X) = \bb{A},
\eew
as required.

\end{proof}

Now return to \pref{A-Ample}.  The proof is broken into three stages. Initially, define the larger set:
\ew
\mc{N}(\rh;\bb{B})_0 = \lt\{\nu \in \ww{2}\bb{B}^* ~\middle|~\th \w \nu + \rh \in \ww[+]{3}\lt(\bb{R}^6\rt)^*\rt\} \pc \ww{2}\bb{B}^*.
\eew
Since \(\Xi\) is generic for \(\rh\) and \(\bb{B} \in \Xi\), it follows that \(\ta = \rh|_\bb{B}\) is an ospseudoplectic form on \(\bb{B}\).  Noting that \(\mc{N}(\rh;\bb{B})_0\) is just a translated copy of \(\mc{N}_{\rh_+}(\ta)\), by \pref{slr-setup} it follows that \(\mc{N}(\rh;\bb{B})_0 \pc \ww{2}\bb{B}^*\) is ample (and, indeed, path-connected).  For each \(\bb{B}'\in\Xi\) define a closed subset \(\Si_{\bb{B}'} \pc \mc{N}(\rh;\bb{B})_0\) by:
\ew
\Si_{\bb{B}'} = \lt\{\nu \in \mc{N}(\rh;\bb{B})_0 ~\m|~ \bb{B}' \text{ is not generic for } \th \w \nu + \rh\rt\}
\eew
and define:
\ew
\mc{N}(\rh;\Xi,\bb{B})_1 = \lt.\mc{N}(\rh;\bb{B})_0\m\osr\bigcup_{\bb{B}' \in \Xi} \Si_{\bb{B}'}\rt..
\eew
Explicitly:
\ew
\mc{N}(\rh;\Xi,\bb{B})_1 = \lt\{\nu \in \ww{2}\bb{B}^* ~\m|~\th \w \nu + \rh \in \ww[+]{3}\lt(\bb{R}^6\rt)^* \text{ and every \(\bb{B}' \in \Xi\) is generic for } \th \w \nu + \rh\rt\}.
\eew
Next, for each pair \(\{\bb{B}',\bb{B}''\} \cc \Xi\) define closed subsets \(\Si^\pm_{\{\bb{B}',\bb{B}''\}} \pc \mc{N}(\rh;\Xi,\bb{B})_1\) by:
\ew
\Si^+_{\{\bb{B}',\bb{B}''\}} = \lt\{\nu \in \mc{N}(\rh;\Xi,\bb{B})_1 ~\m|~ \bb{B}' \cap E_{\pm, \th \w \nu + \rh} = \bb{B}'' \cap E_{\pm, \th \w \nu + \rh} \text{ and } \bb{B}' \cap E_{+, \th \w \nu + \rh} = \bb{B} \cap E_{+,\th \w \nu + \rh} \rt\}
\eew
and
\ew
\Si^-_{\{\bb{B}',\bb{B}''\}} = \lt\{\nu \in \mc{N}(\rh;\Xi,\bb{B})_1 ~\m|~ \bb{B}' \cap E_{\pm, \th \w \nu + \rh} = \bb{B}'' \cap E_{\pm, \th \w \nu + \rh} \text{ and } \bb{B}' \cap E_{-, \th \w \nu + \rh} = \bb{B} \cap E_{-,\th \w \nu + \rh} \rt\},
\eew
and set:
\ew
\mc{N}(\rh;\Xi,\bb{B})_2 = \lt.\mc{N}(\rh;\Xi,\bb{B})_1\m\osr\bigcup_{\{\bb{B}',\bb{B}''\} \cc \Xi} \lt(\Si^+_{\{\bb{B}',\bb{B}''\}} \cup \Si^-_{\{\bb{B}',\bb{B}''\}}\rt)\rt..
\eew
Explicitly:
\ew
\mc{N}(\rh;\Xi,\bb{B})_2 = \bigg\{ \nu \in \mc{N}(\rh;\Xi,\bb{B})_1 ~\bigg|~ \raisebox{4pt}{\parbox{71mm}{\center{if \(\{\bb{B}',\bb{B}''\} \cc \Xi\) is non-generic for \(\rh' = \th \w \nu + \rh\), then \(\bb{B}' \cap E_{\pm,\rh'} \ne \bb{B} \cap E_{\pm,\rh'}\)}}} \bigg\}.
\eew
Finally, for each pair \(\{\bb{B}',\bb{B}''\} \cc \Xi\) define a closed subset \(\Si_{\{\bb{B}',\bb{B}''\}} \pc \mc{N}(\rh;\Xi,\bb{B})_2\) by:
\ew
\Si_{\{\bb{B}',\bb{B}''\}} = \lt\{\nu \in \mc{N}(\rh;\Xi,\bb{B})_2 ~\m|~ \bb{B}' \cap E_{\pm, \th \w \nu + \rh} = \bb{B}'' \cap E_{\pm, \th \w \nu + \rh}\rt\}.
\eew
Set:
\ew
\mc{N}(\rh;\Xi,\bb{B})_3 = \lt.\mc{N}(\rh;\Xi,\bb{B})_2\m\osr\bigcup_{\{\bb{B}',\bb{B}''\} \cc \Xi} \Si_{\{\bb{B}',\bb{B}''\}}\rt.
\eew
and observe that, by construction, \(\mc{N}(\rh;\Xi,\bb{B})_3 = \mc{N}(\rh;\Xi,\bb{B})\).  Thus, by applying \lref{ample-vs-mac} three times, to prove \pref{A-Ample} it suffices to prove the following lemma:
\begin{Lem}\label{mac-lem}~
\begin{enumerate}
\item For all \(\bb{B}'\in\Xi\), the subset \(\Si_{\bb{B}'} \pc \mc{N}(\rh;\bb{B})_0\) is macilent.

\item For all \(\{\bb{B}',\bb{B}''\} \cc \Xi\), the subsets \(\Si^\pm_{\{\bb{B}',\bb{B}''\}} \pc \mc{N}(\rh;\Xi,\bb{B})_1\) are macilent.

\item For all \(\{\bb{B}',\bb{B}''\} \cc \Xi\), the subset \(\Si_{\{\bb{B}',\bb{B}''\}} \pc \mc{N}(\rh;\Xi,\bb{B})_2\) is macilent.
\end{enumerate}
\end{Lem}

\noindent The proof of this result occupies the rest of this paper.\\

\section{Computing the derivatives of \(\rh \mt E_{\pm,\rh}\)}\label{Epm-Deriv-Sec}

Given \(\rh\in\ww[+]{3}\lt(\bb{R}^6\rt)^*\), recall that there is a decomposition \(\bb{R}^6 = E_{+,\rh} \ds E_{-,\rh}\).  Thus, there is also a decomposition:
\ew
\ww{p}\lt(\bb{R}^6\rt)^* \cong \Ds_{r+s = p} \ww{r}E_{+,\rh}^* \ts \ww{s}E_{-,\rh}^* = \Ds_{r+s = p} \ww{r,s}\lt(\bb{R}^6\rt)^*.
\eew
Define \(\SL(3;\bb{R})^2\)-equivariant isomorphisms \(\ka^+_\rh: \ww{2,0}\lt(\bb{R}^6\rt)^* \to E_{+,\rh}\) and \(\ka^-_\rh: \ww{0,2}\lt(\bb{R}^6\rt)^* \to E_{-,\rh}\) as the inverses to the maps:
\ew
\bcd[row sep = -3mm, column sep = 7mm]
E_{+,\rh} \ar[r]& \ww{2,0}\lt(\bb{R}^6\rt)^* & \hs{-4mm}\raisebox{-2.5mm}{and}\hs{-4mm} & E_{-,\rh} \ar[r]& \ww{0,2}\lt(\bb{R}^6\rt)^* & \hs{-4mm}\raisebox{-2.5mm}{respectively.}\\
w \ar[r, maps to]& w \hk \rh & &  w \ar[r, maps to]& w \hk \rh &
\ecd
\eew
\begin{Prop}\label{Epm-deriv}
Consider the smooth maps:
\ew
\bcd[row sep = 0pt, column sep = 7mm]
E_\pm:\ww[+]{3}\lt(\bb{R}^6\rt)^* \ar[r]& \Gr_3\lt(\bb{R}^6\rt)\\
\rh \ar[r, maps to] & E_{\pm,\rh}.
\ecd
\eew
Fix \(\rh \in \ww[+]{3}\lt(\bb{R}^6\rt)^*\).  Then:
\ew
\bcd[row sep = 0pt, column sep = 7mm]
\mc{D}E_+|_\rh: \ww{3}\lt(\bb{R}^6\rt)^* \ar[r]& \lt(E_{+,\rh}\rt)^*\ts E_{-,\rh} \cong \Hom(E_{+,\rh},E_{-,\rh})\\
\al \ar[r, maps to]& -(\Id\ts\ka^-_\rh)(\pi_{1,2}(\al))
\ecd
\eew
and
\ew
\bcd[row sep = 0pt, column sep = 7mm]
\mc{D}E_-|_\rh: \ww{3}\lt(\bb{R}^6\rt)^* \ar[r]& E_{+,\rh}\ts\lt(E_{-,\rh}\rt)^* \cong \Hom(E_{-,\rh},E_{+,\rh})\\
\al \ar[r, maps to]& (\ka^+_\rh\ts\Id)(\pi_{2,1}(\al)),
\ecd
\eew
respectively, where \(\pi_{r,s}\) denotes the projection onto forms of type \((r,s)\).
\end{Prop}

\begin{proof}
Start with the first statement.  Since \(\ww[+]{3}\lt(\bb{R}^6\rt)^* \pc \ww{3}\lt(\bb{R}^6\rt)^*\) is open, one has \(\T_\rh\ww[+]{3}\lt(\bb{R}^6\rt)^* = \ww{3}\lt(\bb{R}^6\rt)^*\).  Likewise, the decomposition \(\bb{R}^6 = \bb{E}_{+,\rh} \ds E_{-,\rh}\) yields \(\T_{E_{+,\rh}}\Gr_3\lt(\bb{R}^6\rt) \cong \Hom(E_{+,\rh},E_{-,\rh})\).  Since the only simple \(\SL(3;\bb{R})^2\)-submodule of \(\ww{3}\lt(\bb{R}^6\rt)^*\) which is isomorphic to \(\Hom(E_{+,\rh},E_{-,\rh}) \cong \lt(E_{+,\rh}\rt)^*\ts E_{-,\rh}\) is \(\ww{1,2}\lt(\bb{R}^6\rt)^*\), it follows that:
\ew
\mc{D}E_+|_\rh(\al) = C\Id\ts\ka^-_\rh(\pi_{1,2}(\al))
\eew
for some constant \(C\).

The value of \(C\) may be computed directly.  Consider \(\rh = \rh_+ = \th^{123} + \th^{456}\) and write:
\ew
\rh_t = \rh_+ + t\th^{145}.
\eew
A direct calculation shows that:
\ew
E_{+,\rh_t} = \<e_1 - te_6, e_2,e_3\?
\eew
so that:
\ew
\lt.\frac{\dd}{\dd t}E_{+,\rh_t}\rt|_{t=0} = -\th^1 \ts e_6.
\eew
By  comparison:
\ew
(\Id\ts\ka^-_{\rh_+})(\pi_{1,2}(\th^{145})) = \th^1 \ts e_6,
\eew
forcing \(C=-1\), as claimed.  The calculation for \(\mc{D}E_-|_\rh\) is similar.

\end{proof}
~

\section{\lref{mac-lem}(1): the macilence of \(\Si_{\bb{B}'}\)}

Recall the set:
\ew
\mc{N}(\rh;\bb{B})_0 = \lt\{\nu \in \ww{2}\bb{B}^* ~\m|~\th \w \nu + \rh \in \ww[+]{3}\lt(\bb{R}^6\rt)^*\rt\} \pc \ww{2}\bb{B}^*
\eew
and also the closed subset:
\ew
\Si_{\bb{B}'} = \lt\{\nu \in \mc{N}(\rh;\bb{B})_0~\m|~\bb{B}' \text{ is not generic for } \th \w \nu + \rh\rt\}.
\eew

\begin{Lem}\label{gen-stab-1}
\ew
\Si_\bb{B} = \es.
\eew
\end{Lem}

\begin{proof}
Indeed, let \(\nu \in \mc{N}(\rh; \bb{B})_0\), i.e.\ suppose that \(\th \w \nu + \rh\) is an \slr\ 3-form.  Then:
\ew
(\th \w \nu + \rh)|_\bb{B} = \rh|_\bb{B}.
\eew
Since \(\bb{B}\) is generic for \(\rh\), \(\rh|_\bb{B}\) is an ospseudoplectic 3-form and thus \(\bb{B}\) must also be generic for \(\th \w \nu + \rh\) (else \((\th \w \nu + \rh)|_\bb{B}\) would be decomposable).

\end{proof}

\begin{Rk}
The above proof also shows that if \(\bb{B}\) is non-generic for \(\rh\) (equivalently, if \(\rh|_\bb{B}\) is decomposable) then it is also non-generic for all \(\th \w \nu + \rh\).  At first sight, this result may seem surprising, since one expects non-genericity to be destroyed by pertubations. On closer examination, however, the result is less surprising, since the space of perturbations of \(\rh\) of the form \(\th \w \nu + \rh\) is \(\bin{5}{2} = 10\)-dimensional, whereas the space of all perturbations of \(\rh\) is instead \(\bin{6}{3} = 20\)-dimensional.
\end{Rk}

\begin{Lem}\label{gen-stab-2}
Let \(\nu \in \mc{N}(\rh;\bb{B})_0\) and write \(\rh' = \th \w \nu + \rh \in \ww[+]{3}\lt(\bb{R}^6\rt)^*\).  Then:
\ew
(\bb{B} \cap E_{+,\rh}) \ds (\bb{B} \cap E_{-,\rh}) = (\bb{B} \cap E_{+,\rh'}) \ds (\bb{B} \cap E_{-,\rh'}).
\eew
\end{Lem}

\begin{proof}
By applying a suitable orientation-preserving automorphism of \(\bb{R}^6\) one can always assume that:
\ew
\rh = \th^{123} + \th^{456} \et \bb{B} = \<e_1,e_2,e_4,e_5,e_3 + e_6\?.
\eew
Hence:
\e\label{sum-of-int}
(\bb{B} \cap E_{+,\rh}) \ds (\bb{B} \cap E_{-,\rh}) = \<e_1,e_2\? \ds \<e_4,e_5\? = \<e_1,e_2,e_4,e_5\?.
\ee
Now, take \(\bb{L} = \<e_3 - e_6\?\), \(\th = \th^3-\th^6\) and write:
\ew
\rh' = \th^{123} + \th^{456} + (\th^3 - \th^6) \w \nu.
\eew
Recall the para-complex structure \(I_{\rh'}\) induced by \(\rh'\).

\begin{Cl}
\ew
I_{\rh'}\lt(\<e_1,e_2,e_4,e_5\?\rt) \cc \<e_1,e_2,e_4,e_5\?.
\eew
\end{Cl}

\begin{proof}[Proof of Claim]
Recall the map:
\ew
\bcd[row sep = 0]
\fr{i}_{\rh'}:\bb{R}^6 \ar[r] & \ww{5}\lt(\bb{R}^6\rt)^*\\
~v \ar[r, maps to] & (v \hk \rh') \w \rh'.
\ecd
\eew
Then, by the definition of \(I_{\rh'}\), it is equivalent to prove that:
\ew
\fr{i}_{\rh'}\lt(\<e_1,e_2,e_4,e_5\?\rt) \cc \th^{36} \w \ww{3}\lt(\bb{R}^6\rt)^*.
\eew
Consider the subgroup \(\SL(2;\bb{R})^2 \pc \SL(3;\bb{R})^2\) acting block diagonally on \(\<e_1,e_2\? \ds \<e_4,e_5\?\) and trivially on \(\<e_3,e_6\?\).  Clearly, \(\SL(2;\bb{R})^2\) preserves \(\rh\), \(\bb{B}\), \(\bb{L}\) and \(\th\) as described above, and acts transitively on the set of non-zero vectors in both \(\<e_1,e_2\?\) and \(\<e_4,e_5\?\).  By exploiting this freedom, it suffices to prove that:
\ew
\fr{i}_{\rh'}(e_1), \fr{i}_{\rh'}(e_4) \in \th^{36} \w \ww{3}\lt(\bb{R}^6\rt)^*.
\eew
However, a direct calculation yields:
\ew
(e_1 \hk \rh') \w \rh' &= \lt(\th^{23} - \th^3 \w (e_1 \hk \nu) + \th^6 \w (e_1 \hk \nu)\rt) \w \lt(\th^{123} + \th^{456} + (\th^3 - \th^6) \w \nu\rt)\\
&= \lt(\th^{245} - \th^2 \w \nu + \th^{12} \w (e_1 \hk \nu) + \th^{45} \w (e_1 \hk \nu)\rt) \w \th^{36}
\eew
whilst:
\ew
(e_4 \hk \rh') \w \rh' &= \lt(\th^{56} - \th^3 \w (e_4 \hk \nu) + \th^6 \w (e_4 \hk \nu)\rt) \w \lt(\th^{123} + \th^{456} + (\th^3 - \th^6) \w \nu\rt)\\
&= \lt(-\th^{125} - \th^5 \w \nu + \th^{45} \w (e_4 \hk \nu) + \th^{12} \w (e_4 \hk \nu)\rt) \w \th^{36},
\eew
completing the proof of the claim.

\let\qed\relax
\end{proof}

Using the claim, \(\lt(I_{\rh'}|_{\<e_1e_2,e_4,e_5\?}\rt)^2 = \Id\) and thus:
\ew
\<e_1,e_2,e_4,e_5\? = e_+ \ds e_-,
\eew
where \(e_\pm\) are the \(\pm1\)-eigenspaces of \(I_{\rh'}|_{\<e_1,e_2,e_4,e_5\?}\).  Since \(\<e_1,e_2,e_4,e_5\? \pc \bb{B}\), it follows that \(e_\pm \cc \bb{B} \cap E_{\pm,\rh'}\) and hence:
\ew
\<e_1,e_2,e_4,e_5\? = e_+ \ds e_- \cc \lt(\bb{B} \cap E_{+,\rh'}\rt) \ds \lt(\bb{B} \cap E_{-,\rh'}\rt).
\eew
However, \(\bb{B}\) is generic for \(\rh'\) by \lref{gen-stab-1} and hence:
\ew
\dim\lt[\lt(\bb{B} \cap E_{+,\rh'}\rt) \ds \lt(\bb{B} \cap E_{-,\rh'}\rt)\rt] = 4.
\eew
Therefore (see \eref{sum-of-int}):
\ew
\lt(\bb{B} \cap E_{+,\rh'}\rt) \ds \lt(\bb{B} \cap E_{-,\rh'}\rt) = \<e_1,e_2,e_4,e_5\? = (\bb{B} \cap E_{+,\rh}) \ds (\bb{B} \cap E_{-,\rh}),
\eew
as required.

\end{proof}

\begin{Lem}\label{non-gen-stab-2}
Let \(\nu \in \mc{N}(\rh;\bb{B})_0\) and write \(\rh' = \th \w \nu + \rh \in \ww[+]{3}\lt(\bb{R}^6\rt)^*\).  Suppose a hyperplane \(\bb{B}' \ne \bb{B}\) satisfies:
\e\label{deg-pair}
\bb{B} \cap E_{+,\rh'} \cc \bb{B}'\cap E_{+,\rh'} \et \bb{B} \cap E_{-,\rh'} \cc \bb{B}'\cap E_{-,\rh'}.
\ee
Then, \eref{deg-pair} also holds \wrt\ \(\rh\), i.e.:
\e\label{orig-deg}
\bb{B} \cap E_{+,\rh} \cc \bb{B}'\cap E_{+,\rh} \et \bb{B} \cap E_{-,\rh} \cc \bb{B}'\cap E_{-,\rh}.
\ee
In particular, \(\{\bb{B},\bb{B}'\}\) is non-generic for \(\rh\).
\end{Lem}

\begin{proof}
Firstly, note that:
\ew
\bb{B} \cap E_{\pm,\rh} &= \lt[\lt(\bb{B} \cap E_{+,\rh}\rt) \ds \lt(\bb{B} \cap E_{-,\rh}\rt)\rt] \cap E_{\pm,\rh}\\
&= \lt[\lt(\bb{B} \cap E_{+,\rh'}\rt) \ds \lt(\bb{B} \cap E_{-,\rh'}\rt)\rt] \cap E_{\pm,\rh} \hs{3mm} \text{ by \lref{gen-stab-2}}\\
&\cc \lt[\lt(\bb{B}' \cap E_{+,\rh'}\rt) \ds \lt(\bb{B}' \cap E_{-,\rh'}\rt)\rt] \cap E_{\pm,\rh} \hs{3mm} \text{ by \eref{deg-pair}}\\
&\cc \bb{B}' \cap E_{\pm,\rh},
\eew
as required.  For the final statement, note that either \(\bb{B}'\) itself is non-generic for \(\rh\), or else \(\dim(\bb{B}'\cap E_{+,\rh}) = \dim(\bb{B}'\cap E_{-,\rh}) = 2\) together with \eref{orig-deg} forces:
\ew
\bb{B}\cap E_{+,\rh} = \bb{B}'\cap E_{+,\rh} \et \bb{B}\cap E_{-,\rh} = \bb{B}'\cap E_{-,\rh}.
\eew
In either case, \(\lt\{\bb{B},\bb{B}'\rt\}\) is non-generic for \(\rh\).

\end{proof}

\begin{Rk}
If both \(\bb{B}\) and \(\bb{B}'\) are individually generic for \(\rh\), it is clear that \(\{\bb{B},\bb{B}'\}\) is non-generic for \(\rh\) \iff\ \eref{orig-deg} is satisfied.
\end{Rk}

I now prove \lref{mac-lem}(1).  Recall the statement of the lemma:\vs{3mm}

\noindent{\bf Lemma \ref{mac-lem}(1).}
{\it For all \(\bb{B}'\in\Xi\), the subset \(\Si_{\bb{B}'} \pc \mc{N}(\rh;\bb{B})_0\) is macilent. More precisely, it is either empty or the disjoint union of two closed submanifolds, each of codimension 3.}\vs{2mm}

\begin{proof}
By \lref{gen-stab-1}, it suffices to consider \(\bb{B}' \ne \bb{B}\).  Consider the maps:
\ew
\bcd[row sep = 0pt]
\bb{E}_\pm:\mc{N}(\rh;\bb{B})_0 \ar[r]& \Gr_3\lt(\bb{R}^6\rt)\\
\nu \ar[r, maps to]& E_{\pm,\th \w \nu + \rh}\hs{1.5pt}.
\ecd
\eew
(I use the notation \(\bb{E}_\pm\) to emphasise that, unlike the maps \(E_\pm\), the arguments of the maps \(\bb{E}_\pm\) are 2-forms, and not \slr\ 3-forms.)  Consider the submanifold \(\Gr_3(\bb{B}') \pc \Gr_3\lt(\bb{R}^6\rt)\) and recall that \(\bb{B}'\) is non-generic for \(\th \w \nu + \rh\) \iff\ either \(\bb{E}_+(\nu)\) or \(\bb{E}_-(\nu)\) lies in \(\Gr_3(\bb{B}')\). Thus:
\ew
\Si_{\bb{B}'} = \lt[\lt(\bb{E}_+\rt)^{-1}\Gr_3(\bb{B}')\rt] \msm{\coprod} \lt[\lt(\bb{E}_-\rt)^{-1}\Gr_3(\bb{B}')\rt].
\eew

\begin{Cl}
The maps \(\bb{E}_\pm\) are transverse to the submanifold \(\Gr_3(\bb{B}')\).
\end{Cl}

\begin{proof}[Proof of Claim]
I consider \(\bb{E}_+\), the case of \(\bb{E}_-\) being essentially identical.  Suppose that \(\nu\in\mc{N}(\rh;\bb{B})_0\) satisfies \(\bb{E}_+(\nu) \in \Gr_3(\bb{B}')\). Write \(\rh' = \th \w \nu + \rh\) and after applying a suitable orientation-preserving automorphism of \(\bb{R}^6\), assume that:
\ew
\rh' = \th^{123} + \th^{456} \et \bb{B}' = \<e_1,e_2,e_3,e_4,e_5\?.
\eew
(Note that there is a residual \(\SL(3;\bb{R})\x\SL(2;\bb{R})\) freedom in choosing such an automorphism, acting diagonally on \(\<e_1,e_2,e_3\? \ds \<e_4,e_5\?\) and trivially on \(\<e_6\?\), a fact which will be exploited below.) Then, one may identify \(\T_{\bb{E}_+(\nu)}\Gr_3(\bb{B}') \cong \Hom(\<e_1,e_2,e_3\?,\<e_4,e_5\?)\) and moreover:
\ew
\rqt{\T_{\bb{E}_+(\nu)}\Gr_3\lt(\bb{R}^6\rt)}{\T_{\bb{E}_+(\nu)}\Gr_3(\bb{B}')} &\cong \rqt{\Hom(\<e_1,e_2,e_3\?,\<e_4,e_5,e_6\?)}{\Hom(\<e_1,e_2,e_3\?,\<e_4,e_5\?)}\\
&\cong \Hom(\<e_1,e_2,e_3\?,\<e_6\?).
\eew

Next recall that \(\Ann(\bb{B}) = \<\th\?\) and write:
\ew
\th = \sum_{i=1}^6 \la_i \th^i = \sum_{i=1}^3 \la_i\th^i + \sum_{i=4}^5 \la_i\th^i + \la_6\th^6.
\eew
By exploiting the residual \(\SL(3;\bb{R})\x\SL(2;\bb{R})\) freedom described above, \wlg\ assume that:
\ew
\th = \la_1\th^1 + \la_4\th^4 + \la_6\th^6.
\eew
I claim that \(\la_4\ne0\). Indeed, suppose \(\th = \la_1\th^1 + \la_6\th^6\). If \(\la_6 = 0\), then \(E_{-,\rh'} = \<e_4,e_5,e_6\? \pc \Ker(\th) = \bb{B}\), hence \(\bb{B}\) is non-generic for \(\rh'\) and whence \(\nu \in \Si_\bb{B}\), contradicting \lref{gen-stab-1}. Thus, \(\la_6\ne0\) and:
\ew
\bb{B} \cap E_{-,\rh'} = \<e_4,e_5\? = \bb{B}' \cap E_{-,\rh'}.
\eew
However, since \(E_{+,\rh'} \pc \bb{B}'\), one trivially has that \(\bb{B} \cap E_{+,\rh'} \cc \bb{B}' \cap E_{+,\rh'}\). Thus, using \lref{non-gen-stab-2}, the pair \(\{\bb{B},\bb{B}'\} \cc \Xi\) is not generic for \(\rh\), which contradicts the assumption that \(\Xi\) is generic for \(\rh\). Thus, \(\la_4 \ne 0\), as claimed.

Finally, note that \(\T_\nu\mc{N}(\rh;\bb{B})_0 = \ww{2}\bb{B}^*\), since \(\mc{N}(\rh;\bb{B})_0 \pc \ww{2}\bb{B}^*\) is open by the stability of \slr\ 3-forms.  Choose \(\nu_i \in \ww{2}\bb{B}^*\) for \(i=1,2,3\) such that:
\ew
\th \w \nu_i = \th \w \th^{i5}.
\eew
(Such \(\nu_i\) exists, since \((\th \w \th^{i5})|_\bb{B} = 0\).)  Then:
\ew
\mc{D}E_+|_{\rh'}(\nu_i) &= -\Id\ts\ka^-_{\rh'}(\pi_{1,2}(\th \w \th^{i5}))\\
&= \la_4\th^i \ts e_6 - \la_6\th^i \ts e_4
\eew
which projects to the element \(\la_4\th^i\ts e_6\) in \(\Hom(\<e_1,e_2,e_3\?,\<e_6\?) \cong \rqt{\T_{\bb{E}_+(\nu)}\Gr_3\lt(\bb{R}^6\rt)}{\T_{\bb{E}_+(\nu)}\Gr_3(\bb{B}')}\). Since \(\la_4 \ne 0\), this proves the surjectivity of the composite:
\ew
\begin{tikzcd}
\ww{2}\bb{B}^* \ar[r, ^^22\mc{D}\bb{E}_+|_{\nu}^^22] & \T_{\bb{E}_+(\nu)}\Gr_3\lt(\bb{R}^6\rt) \ar[r] & \rqt{\T_{\bb{E}_+(\nu)}\Gr_3\lt(\bb{R}^6\rt)}{\T_{\bb{E}_+(\nu)}\Gr_3(\bb{B}')}.
\end{tikzcd}
\eew
Thus, \(\bb{E}_+\) is transverse to \(\Gr_3(\bb{B}')\), completing the proof of the claim.

\let\qed\relax
\end{proof}

Resuming the main proof, since \(\Gr_3(\bb{B}')\) is closed and has codimension \(9 - 6 = 3\) in \(\Gr_3\lt(\bb{R}^6\rt)\), it follows that the submanifolds \(\lt(\bb{E}_+\rt)^{-1}\Gr_3(\bb{B}')\) and \(\lt(\bb{E}_-\rt)^{-1}\Gr_3(\bb{B}')\) of \(\mc{N}(\rh;\bb{B})_0\) are closed and each have codimension 3, and hence:
\ew
\Si_{\bb{B}'} = \lt[\lt(\bb{E}_+\rt)^{-1}\Gr_3(\bb{B}')\rt] \msm{\coprod} \lt[\lt(\bb{E}_-\rt)^{-1}\Gr_3(\bb{B}')\rt]
\eew
is macilent. This completes the proof.

\end{proof}

\section{\lref{mac-lem}(2): the macilence of \(\Si^\pm_{\{\bb{B}',\bb{B}''\}}\)}

Recall the set:
\ew
\mc{N}(\rh;\Xi,\bb{B})_1 = \lt\{\nu \in \ww{2}\bb{B}^* ~\middle|~\th \w \nu + \rh \in \ww[+]{3}\lt(\bb{R}^6\rt)^* \text{ and every \(\bb{B}' \in \Xi\) is generic for } \th \w \nu + \rh\rt\}.
\eew
For each \(\{\bb{B}',\bb{B}''\} \cc \Xi\), recall further the closed subsets \(\Si^\pm_{\{\bb{B}',\bb{B}''\}} \pc \mc{N}(\rh;\Xi,\bb{B})_1\) defined by:
\ew
\Si^+_{\{\bb{B}',\bb{B}''\}} = \lt\{\nu \in \mc{N}(\rh;\Xi,\bb{B})_1 ~\m|~ \bb{B}' \cap E_{\pm, \th \w \nu + \rh} = \bb{B}'' \cap E_{\pm, \th \w \nu + \rh} \text{ and } \bb{B}' \cap E_{+, \th \w \nu + \rh} = \bb{B} \cap E_{+,\th \w \nu + \rh} \rt\}
\eew
and
\ew
\Si^-_{\{\bb{B}',\bb{B}''\}} = \lt\{\nu \in \mc{N}(\rh;\Xi,\bb{B})_1 ~\m|~ \bb{B}' \cap E_{\pm, \th \w \nu + \rh} = \bb{B}'' \cap E_{\pm, \th \w \nu + \rh} \text{ and } \bb{B}' \cap E_{-, \th \w \nu + \rh} = \bb{B} \cap E_{-,\th \w \nu + \rh} \rt\}.
\eew

The aim of this section is to prove \lref{mac-lem}(2).  Recall the statement of the lemma:\vs{3mm}

\noindent{\bf Lemma \ref{mac-lem}(2).}
\em For all \(\{\bb{B}',\bb{B}''\} \cc \Xi\), the subsets \(\Si^\pm_{\{\bb{B}',\bb{B}''\}} \pc \mc{N}(\rh;\Xi,\bb{B})_1\) are macilent.  More precisely, each subset is contained in a submanifold of codimension 2.\vs{1.5mm}\em

\begin{proof}
Since at least one of \(\bb{B}'\) and \(\bb{B}''\) does not equal \(\bb{B}\), \wlg\ assume that \(\bb{B}' \ne \bb{B}\) and note that \(\Si^\pm_{\{\bb{B}',\bb{B}''\}}\) are contained in the sets:
\ew
\Si^\pm_{\bb{B}'} = \lt\{\nu \in \mc{N}(\rh;\Xi,\bb{B})_1 ~\m|~ \bb{B}' \cap E_{\pm, \th \w \nu + \rh} = \bb{B} \cap E_{\pm,\th \w \nu + \rh} \rt\},
\eew
respectively.  Thus, it suffices to prove that the sets \(\Si^\pm_{\bb{B}'} \pc \mc{N}(\rh;\Xi,\bb{B})_1\) are submanifolds of codimension 2 for each \(\bb{B}' \ne \bb{B}\).  Write \(\fr{C} = \bb{B} \cap \bb{B}'\), a 4-dimensional subspace of \(\bb{R}^6\). Using \(\fr{C}\), one may stratify the manifold \(\Gr_3\lt(\bb{R}^6\rt)\) as:
\ew
\Gr_3\lt(\bb{R}^6\rt) = \Si_1 \cup \Si_2 \cup \Si_3,
\eew
where:
\ew
\Si_i = \{E \in \Gr_3\lt(\bb{R}^6\rt)~|~\dim(\fr{C}\cap E) = i\}.
\eew
Explicitly, \(\Si_1\) is the open and dense subset of 3-planes intersecting \(\fr{C}\) transversally, while \(\Si_3 = \Gr_3(\fr{C})\). To understand the submanifold structure on \(\Si_2\), it is useful to describe its tangent space as a subspace of the tangent space of \(\Gr_3\lt(\bb{R}^6\rt)\).  Specifically, fix \(E \in \Si_2\) and write \(\fr{E} = E \cap \fr{C}\). Choose splittings:
\e\label{Tang-dec}
E = \fr{E}^2 \ds \fr{L}^1, \hs{5mm} \fr{C} = \fr{E}^2 \ds \fr{F}^2 \et \bb{R}^6 = \fr{E}^2 \ds \fr{L}^1 \ds \fr{F}^2 \ds \fr{K}^1,
\ee
where the superscripts denote the dimension of the respective subspaces. Then, \(\T_E\Gr_3\lt(\bb{R}^6\rt)\) may be identified with the space:
\ew
\Hom(\fr{E} \ds \fr{L},\fr{F} \ds \fr{K}) \cong \Hom(\fr{E},\fr{F}) \ds \Hom(\fr{E},\fr{K}) \ds \Hom(\fr{L},\fr{F}) \ds \Hom(\fr{L},\fr{K}).
\eew
Using this description, \(\T_E\Si_2\) is given by:
\ew
\T_E\Si_2 \cong \Hom(\fr{E},\fr{F}) \ds \Hom(\fr{L},\fr{F}) \ds \Hom(\fr{L},\fr{K}),
\eew
and hence:
\ew
\rqt{\T_E\Gr_3\lt(\bb{R}^6\rt)}{\T_E\Si_2} \cong \Hom(\fr{E},\fr{K}).
\eew
In particular, the codimension of \(\Si_2\) in \(\Gr_3\lt(\bb{R}^6\rt)\) is \(\dim[\Hom(\fr{E},\fr{K})] = 2\).

Now, consider the smooth maps:
\ew
\bcd[row sep = 0pt]
\bb{E}_\pm: \mc{N}(\rh;\Xi,\bb{B})_1 \ar[r]& \Gr_3\lt(\bb{R}^6\rt)\\
\nu \ar[r, maps to]& E_{\pm,\th \w \nu + \rh}\hs{1.5pt}.
\ecd
\eew
Since \(\fr{C} = \bb{B} \cap \bb{B}'\), one has:
\ew
\bb{E}_+(\nu) \cap \fr{C} = \lt(\bb{E}_+(\nu) \cap \bb{B}\rt) \cap \lt(\bb{E}_+(\nu) \cap \bb{B}'\rt).
\eew
Since both \(\bb{E}_+(\nu) \cap \bb{B}\) and \(\bb{E}_+(\nu) \cap \bb{B}'\) are 2-dimensional, it follows that \(\dim[\bb{E}_+(\nu) \cap \fr{C}] \le 2\), with equality \iff\ \(\bb{E}_+(\nu) \cap \bb{B} = \bb{E}_+(\nu) \cap \bb{B}'\).  Thus, \(\bb{E}_+\lt(\mc{N}(\rh;\Xi,\bb{B})_1\rt) \cc \Si_1 \cup \Si_2\) and:
\ew
\Si^+_{\bb{B}'} = \lt(\bb{E}_+\rt)^{-1}(\Si_2).
\eew
Likewise, \(\Si^-_{\bb{B}'} = \lt(\bb{E}_-\rt)^{-1}(\Si_2)\).  Therefore, to prove that \(\Si^\pm_{\bb{B}'}\) are submanifolds of codimension 2, it suffices to prove that the maps \(\bb{E}_\pm\) are transversal to the submanifold \(\Si_2 \pc \Gr_3\lt(\bb{R}^6\rt)\).

Firstly, consider the case of \(\Si^-_{\bb{B}'}\).  Let \(\nu \in \Si^-_{\bb{B}'}\) and define \(\rh' = \th \w \nu + \rh \in \ww[+]{3}\lt(\bb{R}^6\rt)^*\).  After applying a suitable orientation-preserving automorphism of \(\bb{R}^6\), one may assume that:
\ew
\rh' = \th^{123} + \th^{456} \et \bb{B} = \<e_1,e_2,e_4,e_5,e_3 + e_6\?.
\eew
Since \(\nu \in \Si^-_{\bb{B}'}\) one has \(\bb{B}' \cap E_{-,\rh'} = \bb{B} \cap E_{-,\rh'} = \<e_4,e_5\?\).  If additionally \(\bb{B}' \cap E_{+,\rh'} = \bb{B} \cap E_{+,\rh'}\), then by \lref{non-gen-stab-2} the pair \(\{\bb{B}, \bb{B}'\}\) is non-generic for \(\rh\), contradicting the fact that \(\Xi\) is generic for \(\rh\).  Thus, \(\bb{B}' \cap E_{+,\rh'}\) intersects \(\bb{B} \cap E_{+,\rh'} = \<e_1,e_2\?\) along a 1-dimensional subspace which, by applying a suitable \(\SL(2;\bb{R})\) symmetry to the subspace \(\<e_1,e_2\?\), can be taken to be \(\<e_1\?\).  Therefore, \(\bb{B}' \cap E_{+,\rh'} = \<e_1,r e_2 + e_3\?\) for some \(r \in \bb{R}\).  Now, consider \(F \in \SL(3;\bb{R})^2\) given by:
\ew
(e_1,e_2,e_3,e_4,e_5,e_6) \mt (e_1,e_2,e_3 - r e_2,e_4,e_5,e_6).
\eew
Then, \(F\) preserves \(\rh'\) and \(\bb{B}\) (and hence \(\bb{B}' \cap E_{-,\rh'} = \bb{B} \cap E_{-,\rh'}\)) and maps:
\ew
\<e_1,r e_2 + e_3\? \mt \<e_1,e_3\?.
\eew
Thus, \wlg\ one can take \(\bb{B}' \cap E_{+,\rh'} = \<e_1,e_3\?\).  Therefore:
\ew
\bb{B}' = \<e_1,e_3,e_4,e_5, s e_2 + t e_6\?
\eew
for some \(s,t \in \bb{R}\).  Note that \(s \ne 0\) (as else \(E_{-,\rh'} \pc \bb{B}'\) and so \(\bb{B}'\) is non-generic for \(\rh'\), contradicting \(\nu \in \mc{N}(\rh;\Xi,\bb{B})_1\)) and similarly \(t \ne 0\) (as else \(E_{+,\rh'} \pc \bb{B}'\)).  Thus, by rescaling \(s\) and \(t\), one may assume \wlg\ that \(t = 1\).  Now, consider \(G \in \SL(3;\bb{R})^2\) given by:
\ew
G:(e_1,e_2,e_3,e_4,e_5,e_6) \mt (s e_1,s^{-1} e_2,e_3,e_4,e_5,e_6).
\eew
Then, \(G\) preserves \(\rh'\), \(\bb{B}\) and preserves \(\bb{B}' \cap E_{+,\rh'} = \<e_1,e_3\?\) and maps:
\ew
\<e_1,e_3,e_4,e_5, s e_2 + e_6\? \mt \<s^{-1}e_1, e_3,e_4,e_5, e_2 + e_6\? = \<e_1, e_3,e_4,e_5, e_2 + e_6\?.
\eew
Thus, \wlg\ one can take \(\bb{B}' = \<e_1, e_3,e_4,e_5, e_2 + e_6\?\) and thus:
\ew
\bb{B} \cap \bb{B}' = \<e_1,e_4,e_5, e_2 + e_3 + e_6\?.
\eew
One can then choose:
\ew
\fr{E} = \<e_4,e_5\?, \hs{5mm} \fr{L} = \<e_6\?, \hs{5mm} \fr{F} = \<e_1,e_2 + e_3 + e_6\? \et \fr{K} = \<e_2 - e_3\?.
\eew

The proof now proceeds by direct calculation.  Choose \(\nu_1,\nu_2 \in \ww{2}\bb{B}^*\) such that:
\ew
\th \w \nu_1 = \th \w \th^{14} \et \th \w \nu_2 = \th \w \th^{15}.
\eew
(Such \(\nu_i\) exists, since \((\th \w \th^{14})|_\bb{B} = (\th \w \th^{15})|_\bb{B} = 0\).)  Using the identification:
\e\label{id-1}
\T_{E_{-,\rh'}}\Gr_3\lt(\bb{R}^6\rt) \cong \Hom\lt(E_{-,\rh'},E_{+,\rh'}\rt) = \Hom\lt(\<e_4,e_5,e_6\?,\<e_1,e_2,e_3\?\rt)
\ee
and \pref{Epm-deriv}, and noting that \(\th = \th^3 - \th^6\) (up to rescaling), one computes that:
\ew
\mc{D}\bb{E}_-|_{\nu}(\nu_1) &= \ka^+_\rh\ts\Id(\pi_{2,1}[(\th^3 - \th^6) \w \th^{14}])\\
&= \th^4 \ts e_2
\eew
and:
\ew
\mc{D}\bb{E}_-|_{\nu}(\nu_2) &= \ka^+_\rh\ts\Id(\pi_{2,1}[(\th^3 - \th^6) \w \th^{15}])\\
&= \th^5 \ts e_2.
\eew

Replacing the identification in \eref{id-1} with the identification:
\ew
\T_{E_{-,\rh'}}\Gr_3\lt(\bb{R}^6\rt) = \Hom(\fr{E} \ds \fr{L}, \fr{F} \ds \fr{K}) = \Hom(\<e_4,e_5,e_6\?, \<e_1,e_2 - e_3, e_2 + e_3 + e_6\?)
\eew
the above results become:
\ew
\mc{D}\bb{E}_-|_{\nu}(\nu_1) = \th^4 \ts \lt(e_2 + \frac{1}{2}e_6\rt) \et \mc{D}\bb{E}_-|_{\nu}(\nu_2) = \th^5 \ts \lt(e_2 + \frac{1}{2}e_6\rt)
\eew
and hence:
\ew
\mc{D}\bb{E}_-\lt(\T_\nu\mc{N}(\rh;\Xi,\bb{B})_1\rt) \yy \Hom\lt(\<e_4,e_5\?, \lt\<e_2 + \frac{1}{2}e_6\rt\?\rt).
\eew
Thus:
\ew
\mc{D}\bb{E}_-\lt(\T_\nu\mc{N}(\rh;\Xi,\bb{B})_1\rt) + \T_{E_{-,\rh'}}\Si_2 &\yy \Hom\lt(\<e_4,e_5\?, \lt\<e_2 + \frac{1}{2}e_6\rt\?\rt) + \Hom(\fr{E},\fr{F})\\
& \hs{3cm} + \Hom(\fr{L},\fr{F}) + \Hom(\fr{L},\fr{K}).
\eew
Substituting the formulae for \(\Hom(\fr{E},\fr{F})\), \(\Hom(\fr{L},\fr{F})\) and \(\Hom(\fr{L},\fr{K})\), it follows that:
\ew
\mc{D}\bb{E}_-\lt(\T_\nu\mc{N}(\rh;\Xi,\bb{B})_1\rt) + \T_{E_{-,\rh'}}\Si_2 \yy \Hom(\<e_4,e_5,e_6\?, \<e_1,e_2 - e_3, e_2 + e_3 + e_6\?) = \T_{E_{-,\rh'}}\Gr_3\lt(\bb{R}^6\rt).
\eew
Thus, \(\bb{E}_-\) is transverse to \(\Si_2\), as required.\\

The case of \(\Si^+_{\bb{B}'}\) is analogous.  In a similar fashion to above, one argues that \wlg:
\ew
\rh' = \th^{123} + \th^{456}, \hs{5mm} \bb{B} = \<e_1,e_2,e_4,e_5,e_3 + e_6\?, \hs{5mm} \bb{B}' = \<e_1,e_2,e_4,e_6, e_3 + e_5\? \et \th = \th^3 - \th^6,
\eew
takes:
\ew
\fr{E} = \<e_1,e_2\?, \hs{5mm} \fr{L} = \<e_3\?, \hs{5mm} \fr{F} = \<e_4,e_3 + e_5 + e_6\? \et \fr{K} = \<e_5 - e_6\?
\eew
and identifies:
\ew
\T_{E_{+,\rh'}}\Gr_3\lt(\bb{R}^6\rt) = \Hom(\fr{E} \ds \fr{L}, \fr{F} \ds \fr{K}) = \Hom(\<e_1,e_2,e_3\?, \<e_4,e_5 - e_6, e_3 + e_5 + e_6\?).
\eew
By considering the derivative in the \(\nu_1\) and \(\nu_2\) directions, where \(\th \w \nu_1 = \th \w \th^{14}\) and \(\th \w \nu_2 = \th \w \th^{24}\), one verifies that:
\ew
\mc{D}\bb{E}_+\lt(\T_\nu\mc{N}(\rh;\Xi,\bb{B})_1\rt) \yy \Hom\lt(\<e_1,e_2\?,\lt\<\frac{1}{2}e_3 + e_5\rt\?\rt)
\eew
from which the result follows.

\end{proof}

\section{\lref{mac-lem}(3): the macilence of \(\Si_{\{\bb{B}',\bb{B}''\}}\)}\label{3rd-Mac-Sec}

Recall the set:
\ew
\mc{N}(\rh;\Xi,\bb{B})_2 = \bigg\{ \nu \in \mc{N}(\rh;\Xi,\bb{B})_1 ~\bigg|~ \raisebox{4pt}{\parbox{71mm}{\center{if \(\{\bb{B}',\bb{B}''\} \cc \Xi\) is non-generic for \(\rh' = \th \w \nu + \rh\), then \(\bb{B}' \cap E_{\pm,\rh'} \ne \bb{B} \cap E_{\pm,\rh'}\)}}} \bigg\}.
\eew
For each \(\{\bb{B}',\bb{B}''\} \cc \Xi\), recall further the closed subset \(\Si_{\{\bb{B}',\bb{B}''\}} \pc \mc{N}(\rh;\Xi,\bb{B})_2\) defined by:
\ew
\Si_{\{\bb{B}',\bb{B}''\}} = \lt\{\nu \in \mc{N}(\rh;\Xi,\bb{B})_2 ~\m|~ \bb{B}' \cap E_{\pm, \th \w \nu + \rh} = \bb{B}'' \cap E_{\pm, \th \w \nu + \rh}\rt\}.
\eew

\begin{Lem}\label{es-lem}
For all \(\{\bb{B},\bb{B}'\} \cc \Xi\):
\ew
\Si_{\{\bb{B},\bb{B}'\}} = \es.
\eew
\end{Lem}

\begin{proof}
Indeed, if there were \(\nu \in \Si_{\{\bb{B},\bb{B}'\}}\) then, writing \(\rh' = \th \w \nu + \rh \in \ww[+]{3}\lt(\bb{R}^6\rt)^*\), one would find \(\bb{B} \cap E_{\pm,\rh'} = \bb{B}' \cap E_{\pm,\rh'}\) and thus, by \lref{non-gen-stab-2}, it would follow that \(\{\bb{B},\bb{B}'\} \cc \Xi\) was not generic for \(\rh\), contradicting the fact that \(\Xi\) is generic for \(\rh\).

\end{proof}

I now prove \lref{mac-lem}(3).  Recall the statement of the lemma:\vs{3mm}

\noindent{\bf Lemma \ref{mac-lem}(3).}
\em For all \(\{\bb{B}',\bb{B}''\} \cc \Xi\), the subset \(\Si_{\{\bb{B}',\bb{B}''\}} \pc \mc{N}(\rh;\Xi,\bb{B})_2\) is macilent.\vs{2mm}\em

\begin{proof}
By \lref{es-lem}, \wlg\ assume that \(\bb{B}' \ne \bb{B} \ne \bb{B}''\).  Since \(\bb{B}' \ne \bb{B}''\), defining \(\fr{C}' = \bb{B}' \cap \bb{B}''\) one finds, as in the proof of \lref{mac-lem}(2), that \(\fr{C}' \pc \bb{R}^6\) is 4-dimensional and induces a stratification:
\ew
\Gr_3\lt(\bb{R}^6\rt) = \Si'_1 \cup \Si'_2 \cup \Si'_3
\eew
where:
\ew
\Si'_i = \{E \in \Gr_3\lt(\bb{R}^6\rt)~|~\dim(\fr{C}' \cap E) = i\}.
\eew
Consider the map:
\ew
\bb{E}_+ : \mc{N}(\rh;\Xi,\bb{B})_2 &\to \Gr_3\lt(\bb{R}^6\rt)\\
\nu &\mt E_{+,\th \w \nu + \rh}.
\eew
Since \(\fr{C}' = \bb{B}' \cap \bb{B}''\), one has:
\e\label{dim-force}
\bb{E}_+(\nu) \cap \fr{C}' = \lt(\bb{E}_+(\nu) \cap \bb{B}'\rt) \cap \lt(\bb{E}_+(\nu) \cap \bb{B}''\rt).
\ee
Since both \(\bb{E}_+(\nu) \cap \bb{B}'\) and \(\bb{E}_+(\nu) \cap \bb{B}''\) are 2-dimensional, it follows that \(\dim[\bb{E}_+(\nu) \cap \fr{C}'] \le 2\), with equality \iff\ \(\bb{E}_+(\nu) \cap \bb{B}' = \bb{E}_+(\nu) \cap \bb{B}''\).  Thus, \(\bb{E}_+\lt(\mc{N}(\rh;\Xi,\bb{B})_2\rt) \cc \Si'_1 \cup \Si'_2\) and:
\ew
\Si_{\{\bb{B}',\bb{B}''\}} \cc \lt(\bb{E}_+\rt)^{-1}\lt(\Si'_2\rt).
\eew
(Likewise, \(\Si_{\{\bb{B}',\bb{B}''\}} \cc \lt(\bb{E}_-\rt)^{-1}\lt(\Si'_2\rt)\), a fact which will prove useful below.)  Since \(\Si'_2\) has codimension 2 in \(\Gr_3\lt(\bb{R}^6\rt)\), to complete the proof it suffices to prove that for all \(\nu \in \Si_{\{\bb{B}',\bb{B}''\}}\) the map \(\bb{E}_+\) is transverse to the submanifold \(\Si'_2 \pc \Gr_3\lt(\bb{R}^6\rt)\) at \(\nu\).  (Note that I do not claim \(\bb{E}_+\) is transverse to \(\Si'_2\) at all points of \(\lt(\bb{E}_+\rt)^{-1}\lt(\Si'_2\rt)\) and thus I do not claim that \(\lt(\bb{E}_+\rt)^{-1}\lt(\Si'_2\rt)\) itself is a submanifold of \(\mc{N}(\rh;\Xi,\bb{B})_2\).  The fact that \(\bb{E}_+\) is transverse to \(\Si'_2\) at (and hence also near) each point of \(\Si_{\{\bb{B}',\bb{B}''\}}\) shows that \(\lt(\bb{E}_+\rt)^{-1}(\Si'_2)\) is a submanifold of codimension 2 near each point of \(\Si_{\{\bb{B}',\bb{B}''\}}\), which is sufficient to establish the macilence of \(\Si_{\{\bb{B}',\bb{B}''\}}\).)

To this end, suppose that \(\nu\in\Si_{\{\bb{B}',\bb{B}''\}}\) and write \(\rh' = \th \w \nu + \rh\).  \Wlg, one may assume that \(\rh' = \th^{123} + \th^{456}\), \(\bb{B} = \<e_1,e_2,e_4,e_5,e_3+e_6\?\) and \(\th = \th^3 - \th^6\).  Recall from \eref{dim-force} that:
\ew
E_{\pm,\rh'} \cap \fr{C}' = E_{\pm,\rh'} \cap \bb{B}' = E_{\pm,\rh'} \cap \bb{B}''.
\eew
Recall moreover that, by definition of \(\mc{N}(\rh;\Xi,\bb{B})_2\), \(E_{\pm,\rh'} \cap \fr{C}' \ne \bb{B} \cap E_{\pm,\rh'}\) for both `\(+\)' and `\(-\)'.  Therefore, \(E_{+,\rh'} \cap \fr{C}'\) must intersect \(\bb{B} \cap E_{+,\rh'} = \<e_1,e_2\?\) in a 1-dimensional subspace, which \wlg\ may be taken to be \(\<e_1\?\). Thus:
\ew
E_{+,\rh'} \cap \fr{C}' = \<e_1,r e_2 + e_3\? \text{ for some } r \in \bb{R}.
\eew
Analogously, one can assume \wlg\ that:
\ew
E_{-,\rh'} \cap \fr{C}' = \<e_4,s e_5 + e_6\? \text{ for some } s \in \bb{R}.
\eew
Since \(\fr{C}'\) is itself 4-dimensional, it follows that:
\ew
\fr{C}' = \<e_1,r e_2 + e_3,e_4,s e_5 + e_6\?.
\eew
Thus, using notation analogous to \eref{Tang-dec}, one has:
\ew
\fr{E}' = \bb{E}_+(\nu) \cap \fr{C}' = \<e_1,r e_2 + e_3\?
\eew
and one may then choose \(\fr{L}',\fr{F}',\fr{K}'\) as:
\ew
\fr{L}' = \<e_2\?, \hs{5mm} \fr{F}' = \<e_4,s e_5 + e_6\? \et \fr{K}' = \<e_5\?.
\eew
Now, choose \(\nu_1,\nu_2 \in \ww{2}\bb{B}^*\) such that:
\ew
\th \w \nu_1 = \th \w \th^{46} \et \th \w \nu_2 = \th \w \th^{14}.
\eew
One may then compute that:
\ew
\mc{D}E_+|_{\rh'}(\th \w \nu_1) &= -\Id\ts\ka^-_{\rh'}(\pi_{1,2}((\th^3 - \th^6) \w \th^{46}))\\
&= \th^3 \ts e_5
\eew
while:
\ew
\mc{D}E_+|_{\rh'}(\th \w \nu) &= -\Id\ts\ka^-_{\rh'}(\pi_{1,2}((\th^3 - \th^6)\w\th^{14}))\\
&= -\th^1 \ts e_5.
\eew
Thus:
\ew
\mc{D}\bb{E}_+\lt(\T_\nu\mc{N}(\rh;\Xi,\bb{B})_2\rt) \supseteq \Hom(\<e_1, e_3\?,\<e_5\?)
\eew
and thus:
\ew
\mc{D}\bb{E}_+\lt(\T_\nu\mc{N}(\rh;\Xi,\bb{B})_2\rt) + \T_{E_{+,\rh'}}\Si_2 &\supseteq \Hom(\<e_1, e_3\?,\<e_5\?) \ds \Hom(\fr{E}',\fr{F})\\
& \hs{3cm} \ds \Hom(\fr{L}',\fr{F}') \ds \Hom(\fr{L}',\fr{K}')\\
&= \Hom(\<e_1,e_2,e_3\?, \<e_4,e_5,e_6\?) = \T_{E_{+,\rh'}}\Gr_3\lt(\bb{R}^6\rt),
\eew
which is the required statement of transversality, completing the proof \lref{mac-lem}(3).

\end{proof}

This completes the proof of \tref{slr-hP-thm}.

\qed

~\vs{5mm}

\noindent Laurence H.\ Mayther\\
University of Cambridge\\
United Kingdom\\
{\it lhm32@cam.ac.uk}

\end{document}